\documentclass[12pt]{tac} 

\usepackage{amssymb}
\usepackage{amsmath}
\usepackage{enumitem}
\usepackage[matrix,arrow,curve,cmtip]{xy}
\usepackage{graphicx}
\usepackage{ifdraft}
\usepackage{autonum}
\usepackage{dsfont}

\usepackage[colorlinks=true]{hyperref}
\hypersetup{allcolors=[rgb]{0.1,0.1,0.4}}

\author{Min Liu, Shengwei Han and Isar Stubbe}

\thanks{Corresponding author: Min Liu}

\address{School of Sciences, Chang'an University \\ 710064, Xi'an, China\\[5pt] School of Mathematics and Statistics, Shaanxi Normal University\\ 710062, Xi'an, China\\[5pt] Laboratoire de Math\'{e}matiques Pures et Appliqu\'{e}es, Universit\'{e} du Littoral\\ 50 rue F.~Buisson, 62228 Calais, France\\}

\title{Ideals and continuity \isar{for}\\
quantaloid-enriched categories}

\copyrightyear{2023}



\eaddress{liumin@chd.edu.cn\CR hansw@snnu.edu.cn\CR isar.stubbe@univ-littoral.fr}

\allowdisplaybreaks

\def\distsign{\begin{picture}(0,0)\put(0,0){\circle{4}}\end{picture}}
\def\dist{\mbox{$\xymatrix@1@C=5mm{\ar@{->}[r]|{\distsign}&}$}}

\def\Q{\mathcal{Q}}
\def\bbA{\mathbb{A}}
\def\bbB{\mathbb{B}}
\def\bbC{\mathbb{C}}
\def\Dist{\mathrm{\bf Dist}}
\def\Cat{\mathrm{\bf Cat}}
\def\P{\mathcal{P}}
\def\colim{\mathrm{colim}}

\def\:{:}
\renewcommand{\sup}{\mathrm{sup}}
\newcommand\nat{\natural}

\ifdraft{
\usepackage[notref,notcite]{showkeys}
\newcommand\isar[1]{{\color{red}{#1}}}
\setlength\marginparwidth{2.5cm}
\newcommand\isarnote[1]{\marginpar{\footnotesize\raggedright\isar{{#1}}}}
}
{
\newcommand\isar[1]{{\color{black}{#1}}}
\setlength\marginparwidth{2.5cm}
\newcommand\isarnote[1]{\marginpar{\footnotesize\raggedright\isar{{}}}}
}

\begin{document}

\maketitle

\begin{abstract}
\isar{We study} ideals \isar{in,} and continuity \isar{of,} quantaloid-enriched categories (${\Q}$-categories for short) as a \isar{`many-valued and many-typed'} generalization of \isar{domain theory}. \isar{Abstractly, for} any \isar{(}saturated\isar{)} class $\Phi$ of presheaves\isar{,} we \isar{define and} study the \isar{$\Phi$-}continuity of ${\Q}$-categories. \isar{Concretely,} \isar{we compute three examples of such} saturated classes of presheaves \isar{--} the class of flat ideals, the class of irreducible ideals and the class of conical ideals \isar{--} which are
proper generalizations of ideals in domain theory.
\end{abstract}

\section{Introduction}\label{S1}

Domain theory \cite{AbramJung,Gie2003} originated from the work of Dana Scott \cite{Scott1970,Scott1972} \isar{as} a formal basis for the semantics of programming languages. \isar{Through the study of special properties of partially ordered sets, it} formalizes the intuitive ideas of approximation and convergence in a very general way. Since America and Rutten's \cite{America} \isar{and Smyth's \cite{Smyth1988}} early work on solving recursive domain equations in the category of metric spaces, \isar{several other papers \cite{Rutten1996,Flagg1997,Wagner1994,Wagner1997,HW2011,Matthews1992,ONeill1998,Waszkiewicz2003}} have been \isar{written} to \isar{compare and unify both} the \isar{posetal} and the \isar{metric} approach to solving domain equations. \isar{This lead to so-called quantitative domain theory, where one typically replaces posets, resp.\ metric spaces, by quantale-enriched categories, and one develops general concepts of approximation and convergence for these. With this paper we want to contribute to this line of research, replacing now quantales by quantaloids.}

\isar{Whereas a \emph{quantale} is a monoid in the symmetric
monoidal closed category {\bf Sup} of complete lattices and
supremum-preserving morphisms, a \emph{quantaloid} is a category
enriched in {\bf Sup}.} Viewing \isar{a quantaloid} ${\Q}$ as a
\isar{(locally posetal)} bicategory \cite{benabou1967}, it is
natural to study categories, functors and distributors enriched in
${\Q}$
\isar{\cite{Walters1981,Street1983a,Stubbe20051,Stubbe20052,Stubbe20053}}.
As a special case of categories enriched in a bicategory,
quantaloid-enriched categories have special properties and important
applications. Walters \cite{Walters1982} \isar{uses
${\Q}$-categories to describe the} topos of sheaves on a small site
\isar{(see also \cite{HeymansStubbe2012}), and} Rosenthal
\cite{Rosenthal1995,Rosenthal1996} \isar{relates ${\Q}$-categories
to} automata theory. \isar{More recently, also (many-valued)
topological concepts have been formalized in the context of
quantaloid-enriched categories
\cite{CleHof2009,Hohle2014,laishen,LiuZhao20131,Shen2016,ShenTao2016,ShenTholen2016,Shenzhang2013}.}

Since quantales are \isar{precisely} quantaloids \isar{with a single object}, complete residuated lattices and frames are all special quantaloids. Thus, quantale-enriched categories \cite{Wagner1997,LaiZhang2007} and $L$-ordered sets \cite{Bel2002,Fan2001,Yao2010} are special quantaloid-enriched categories. The theory of quantaloid-enriched categories can be viewed as a many-typed extension of the theory of quantale-enriched categories. This enables us to interpret some structures that can not be treated in the framework of quantale-enriched categories. For instance, partial metric spaces and probabilistic partial metric spaces can be treated as quantaloid-enriched categories \cite{Stubbe2014,Hohle2011,He2018}, \isar{yet} they cannot be quantale-enriched categories because they cannot be treated with only one type.

\isar{In} \cite{Stubbe2007}, Stubbe \isar{studied} ${\Q}$-categories as crucial mathematical structure \isar{for} a \isar{so-called} dynamic logic as common mathematical foundation for dynamic phenomena in both computer science and physics. In that paper, totally continuous cocomplete ${\Q}$-categories were studied as ``dynamic domains''. However, just as Stubbe claimed in \cite{Stubbe2007}, totally continuous cocomplete ${\Q}$-categories are very strong structures. One can argue that, having abandoned the notion of directedness, their usefulness in computation is rather limited. So it is definitely an interesting project to investigate how a notion of directedness can be brought back again. Over the past decade, research progress in this area is slow \isar{and mostly} restricted to \isar{frame- or quantale-enriched categories}.

In this paper, \isar{we further develop Stubbe's work \cite{Stubbe2007} by defining continuity of ${\Q}$-categories relative to any (saturated) class of presheaves.} \isar{W}e propose several saturated classes of presheaves that are natural generalizations of the concept of ideals in \isar{(posetal)} domain theory. Thus we establish quantitative domain theory based on ${\Q}$-categories\isar{, which} enables us to unify the approach of quantale-enriched categories and the approach of partial metric spaces into one framework.

The contents of the paper are arranged as follows. Section \ref{S2} lists some preliminary notions and results about quantaloid\isar{s} and quantaloid-enriched categories. In Section \ref{S3}, based on a \isar{(}saturated\isar{)} class $\Phi$ of presheaves, $\Phi$-cocompleteness of ${\Q}$-categories is studied. In Section \ref{S4}, \isar{three} saturated classes of presheaves are given\isar{:} the class of all irreducible ideals, the class of all flat ideals\isar{,} and the class of all conical ideals. In Section \ref{S5}, \isar{$\Phi$}-continuity in ${\Q}$-categories is defined and studied as a generalization of continuity in domain theory\isar{, and several examples are computed} to show the non-triviality of the generalization. In Section \ref{S6}, $\Phi$-algebraic ${\Q}$-categories are defined and studied. \isar{We end with a conclusion in Section \ref{S7}.}

\section{\isar{Preliminaries on q}uantaloid-enriched categories}\label{S2}

A \emph{quantaloid} ${\Q}$ \cite{Rosenthal1996} is a category enriched in the symmetric monoidal closed category {\bf Sup} of complete lattices and morphisms that preserve arbitrary \isar{suprema}. In elementary terms, a quantaloid ${\Q}$ is a category for which each \isar{hom}-set ${\Q}(p,q)$ is a complete lattice, in such a way that, for all $u, u_{i}\in {\Q}(p, q)$ and all $v, v_{j}\in {\Q}(q, r)$,
\begin{equation}\label{A0}
\isar{v\circ(\bigvee_{i}u_{i})=\bigvee_{i}(v\circ u_{i})\quad\mbox{ and }\quad(\bigvee_{j}v_{j})\circ u=\bigvee_{j}(v_{j}\circ u).}
\end{equation}
\isar{Given any arrow $u:p\to q$ and object $r$ in $\Q$, both $-\circ u:\Q(q,r)\to\Q(p,r)$ and $u\circ-:\Q(r,p)\to\Q(r,q)$ preserve suprema between complete lattices and thus have a right adjoint; we shall denote these by $(-\circ u)\dashv(-\swarrow u)$ and $(u\circ-)\dashv (u\searrow-)$.}

Next, we shall recall some basic \isar{facts} about quantaloid-enriched categories from \cite{Stubbe20051}. \isar{From now on, ${\Q}$ always denotes a small quantaloid, and ${\Q}_{0}$ is the set of its objects.}

A \emph{$\Q$-typed set} $X$ is a (small) set $X$ together with a mapping $t: X\to \Q_{0}: x\mapsto tx$ assigning each element $x \in X$ to its type $tx \in \Q_{0}$. Throughout the paper, we use $tx$ to denote the type of an element $x$ in a $\Q$-typed set $X$.
\begin{definition}\label{AA} A {\sl$\Q$-enriched category}
\footnote{\isar{The definition for $\Q$-category given here is different from that used
by Stubbe in a series of papers \cite{Stubbe20051,Stubbe20052,Stubbe2007,Stubbe2014}
(but agrees with that in the papers \isar{\cite{laishen,Shen2016,Tao2014})
in the choice for the direction of the hom-arrows.
We refer to \cite[Proposition 3.8]{Stubbe20051}:
what we call a $\Q$-category in this paper is precisely a $\Q^{op}$-category for Stubbe \cite{Stubbe20051}
(which itself is the ``opposite'' of a $\Q$-category in the terminology of \cite{Stubbe20051}).
As a consequence, also the direction of the $\Q$-distributors in this paper (see further) is opposite to that in \cite{Stubbe20051}.
Modulo this duality, all notions presented here (presheaf, colimit, ...) agree with those in \cite{Stubbe20051}.}}}
(or {\sl$\Q$-category} for short) ${\bbA}$ consists of a $\Q$-typed set ${\bbA}_{0}$ \isar{together with}
hom-arrows ${\bbA}(a', a): ta'\to ta$ in $\Q$ for all $a,a'\in {\bbA}_{0}$ satisfying
\begin{equation}\label{}
\isar{{\bbA}(a', a'')\circ {\bbA}(a, a')\leq {\bbA}(a, a'')\quad\mbox{ and }\quad 1_{ta}\leq {\bbA}(a, a)}
\end{equation}
\isar{for all $a,a',a''\in {\bbA}_{0}$}.
\end{definition}
\isarnote{Removed restriction to skeletal cats. Added comments where appropriate further on.} There is a natural underlying preorder on ${{\bbA}_{0}}$ \isar{defined} by
\begin{equation}\label{}
x\leq y\quad\isar{\stackrel{\mathrm{def}}{\Longleftrightarrow}}\quad tx=ty\mbox{ and }1_{tx}\leq {\bbA}(x, y).
\end{equation}
\isar{When both $x\leq y$ and $y\leq x$ hold in $(\bbA_0,\leq)$, then we write $x\cong y$ and say that these are \emph{isomorphic} elements;} and ${\bbA}$ is called \emph{skeletal} when \isar{$x\cong y$ always} implies $x=y$.
A $\Q$-category ${\bbB}$ is a \emph{\isar{full} $\Q$-subcategory} of ${\bbA}$ if ${{\bbB}_{0}}\subseteq {{\bbA}_{0}}$ and ${\bbB}(x, y)={\bbA}(x, y)$ for all $x, y\in {\bbB}_{0}$; \isar{and then $\bbB$ is skeletal whenever $\bbA$ is.}
\begin{definition}\label{AB}\isarnote{Made definition.}
A \emph{$\Q$-functor} $F: {\bbA}\to {\bbB}$ between $\Q$-categories is a type-preserving map $F: {\bbA}_{0} \to {\bbB}_{0}$ such that
\begin{equation}\label{}
{\bbA}(a', a)\leq{\bbB}(Fa', Fa)
\end{equation}
for all $a,a'\in{\bbA}_{0}.$
\end{definition}
\isar{The $\Q$-functor $F\:\bbA\to\bbB$ is \emph{fully faithful} when the inequality in the definition above is in fact an equality (for all $a,a'\in\bbA_0$).}

With the pointwise order on $\Q$-functors, given by
\begin{equation}\label{}
F\leq G: {\bbA}\to {\bbB} \quad\isar{\stackrel{\mathrm{def}}{\Longleftrightarrow}}\quad  Fx\leq Gx \mathrm{\ for\ all\ }x\in {\bbA}_{0},
\end{equation}
\isar{and the obvious composition and identities}, $\Q$-categories and $\Q$-functors constitute a 2-category $\Cat(\Q)$. \isarnote{Included "equivalence" for this is appropriate for non-skeletal cats.}\isar{Following general 2-categorical principles, a} $\Q$-functor $F: {\bbA} \to {\bbB}$ is called \emph{left adjoint} to a $\Q$-functor $G: {\bbB} \to {\bbA}$ in $\Cat(\Q)$ (and then $G$ is called \emph{right adjoint} to $F$), written $F\dashv G$, if $1_{{\bbA}}\leq G \circ F$ and $F\circ G\leq 1_{{\bbB}}$ (or equivalently ${\bbB}(Fx, y) = {\bbA}(x, Gy)$ for all $x\in {{\bbA}_{0}}, y\in {{\bbB}_{0}}$.) \isar{A $\Q$-functor $F: {\bbA}\to {\bbB}$ is an \emph{equivalence} of ${\Q}$-categories if there exists a ${\Q}$-functor $G: {\bbB}\to {\bbA}$ with $GF\cong 1_{\bbA}, FG\cong 1_{\bbB}$\isar{; in} this case, we write ${\bbA}\simeq{\bbB}$.}
\begin{definition}\label{AC}\isarnote{Made this a definition.}
A \emph{$\Q$-distributor} $\phi\:{\bbA}\dist{\bbB}$ between two $\Q$-categories is given by a \isar{matrix} of ${\Q}$-arrows $\left(\phi(x, y): tx\to ty\right)_{x\in {\bbA}_{0}, y\in {\bbB}_{0}}$ such that
\begin{equation}\label{}
\isar{{\bbB}(y', y)\circ \phi(x', y')\circ {\bbA}(x, x')\leq \phi (x, y)}
\end{equation}
for all $x, x'\in {\bbA}_{0}$ and $y, y'\in {\bbB}_{0}$.
\end{definition}
\isar{Two such $\Q$-distributors $\phi:\bbA\dist\bbB$ and $\psi:\bbB\dist\bbC$ compose with the formula}
\begin{equation}\label{}
\isar{(\psi\circ\phi)(x,z)=\bigvee_{y\in\bbB_0}\psi(y,z)\circ\phi(x,y).}
\end{equation}
\isar{The identity $\Q$-distributor on ${\bbA}$ is given by the hom-arrows ${\bbA}: {\bbA}\dist{\bbA}$, and together with} the pointwise local order inherited from $\Q$, $\Q$-categories and $\Q$-distributors \isar{then form} a (large) quantaloid $\Dist(\Q)$. \isar{This means in particular that $\Dist(\Q)$ is itself a 2-category, and so we can speak of left/right adjoint $\Q$-distributors in the obvious manner: for $\phi\:\bbA\dist\bbB$ and $\psi\:\bbB\dist\bbA$ we write $\phi\dashv\psi$ whenever $\bbA\leq\psi\circ\phi$ and $\phi\circ\psi\leq\bbB$. But the quantaloid $\Dist(\Q)$ is also ``closed'', i.e.\ both $\phi\circ-$ and $-\circ\phi$ have right adjoints (denoted by $\phi\searrow-$ and $-\swarrow\phi$ respectively). These relate to the adjoints to composition in $\Q$ in the following manner: for any $\phi\:\bbA\dist\bbB$, $\gamma\:\bbA\dist\bbC$ and $\psi\:\bbC\to\bbB$,}
\begin{equation}\label{AY}
\isar{(\phi\searrow\psi)(c,a)=\bigwedge_{b\in\bbB_0}\phi(a,b)\searrow\psi(c,b)
\quad\mbox{and}\quad
(\phi\swarrow\gamma)(c,b)=\bigwedge_{a\in\bbA_0}\phi(a,b)\swarrow\gamma(a,c).}
\end{equation}
Each $\Q$-functor $F: {\bbA}\to {\bbB}$ \isar{determines} $\Q$-distributors\isar{ $F_{\natural}: {\bbA}\dist{\bbB}$ and $F^{\natural}: {\bbB}\dist{\bbA}$ by}
\begin{equation}\label{AX}
\isar{F_{\natural}(x, y)={\bbB}(Fx, y)\quad\mbox{ and }\quad F^{\natural}(y, x)={\bbB}(y, Fx),}
\end{equation}
which form an adjunction $F_{\natural}\dashv F^{\natural}$ in $\Dist(\Q)$. Thus \isar{we have}
\begin{equation}\label{AD}
F_{\natural}\searrow-=F^{\natural}\circ -\quad\mbox{and}\quad -\swarrow F^{\natural}=-\circ F_{\natural}.
\end{equation}
Particularly, $(1_{\bbA})_{\natural}=(1_{\bbA})^{\natural}={\bbA}$, and for two $\Q$-functors $F: {\bbA}\to {\bbB}$ and $G: {\bbB}\to \mathbb{C},$ we have
\begin{equation}\label{AE}
(G\circ F)_{\natural}=G_{\natural}\circ F_{\natural}\quad\mbox{and}\quad (G\circ F)^{\natural}=F^{\natural}\circ G^{\natural}.
\end{equation}
\isar{That is to say, the assignments $F\mapsto F_\nat$ and $F\to F^\nat$ are (covariantly and contravariantly) 2-functorial from $\Cat(\Q)$ to $\Dist(\Q)$; they thus send adjoint functors to adjoint distributors as well. Note however that $F_\nat=F'_\nat$ (or $F^\nat=F'{}^\nat$) only implies $F\cong F'$, and not necessarily $F=F'$.}

For any ${\Q}$-distributor $\phi: {\bbA}\dist{\bbB}$ and $\Q$-functors $F: \mathbb{X}\to {\bbA}, G: \mathbb{Y}\to {\bbB}$, we have
\begin{equation}\label{AF}
\phi(F-, G-)=G^{\natural}\circ\phi\circ F_{\natural}.
\end{equation}

For each $q\in {\Q}_{0}$, \isar{we write $\{q\}$ for the} discrete $\Q$-category with \isar{a single} object $q$, in which $tq=q$ and $\{q\}(q,q)=1_{q}.$
\begin{definition}\label{AG}\isarnote{Made this a definition.}
A \emph{\isar{(contravariant)} presheaf} (also called a \emph{weight}) with type $q$ on a $\Q$-category ${\bbA}$ is a $\Q$-distributor $\mu: {\bbA}\dist \{q\}$. \isar{The $\Q$-typed set of all} presheaves on ${\bbA}$ constitutes a $\Q$-category ${\P}{\bbA}$ \isar{with hom-arrows} ${\P}{\bbA}(\mu,\mu')=\mu'\swarrow\mu$ for all $\mu,\mu'\in ({\P}{\bbA})_{0}.$
\end{definition}
Dually, the $\Q$-category ${\P}^{\dag}{\bbA}$ of \isar{\emph{covariant presheaves} (also \emph{copresheaves})} \cite{Stubbe20051,Shen2016} on ${\bbA}$ has $\Q$-distributors $\lambda: \{q\}\dist {\bbA}$ as objects and ${\P}^{\dag}{\bbA}(\lambda, \lambda')=\lambda'\searrow \lambda$ \isar{as hom-arrows. Note that both $\P\bbA$ and $\P^{\dag}\bbA$ are skeletal $\Q$-categories.}

\isarnote{Moved this up.}The \emph{(covariant) Yoneda embedding} \isar{on a $\Q$-category $\bbA$ is the $\Q$-functor} $Y_{\bbA}: {\bbA}\to{\P}{\bbA}$ \isar{which} sends each $x\in {\bbA}_{0}$ to \isar{the \emph{representable presheaf} ${\bbA}(-, x):\bbA\dist\{tx\}$}. The Yoneda Lemma \cite[Proposition 6.3]{Stubbe20051} \isar{says} that
\begin{equation}\label{YonedaLemma}
{\P}{\bbA}(Y_{\bbA}x, \mu)=\mu(x)
\end{equation}
for all $x\in {\bbA}_{0},\mu\in ({\P}{\bbA})_{0}$. \isar{For two representable presheaves this implies that
\begin{equation}\label{YonedaFF}
\Big((Y_{\bbA})^\nat\circ (Y_{\bbA})_\nat\Big)(x,y)={\P}{\bbA}(Y_{\bbA}x, Y_{\bbA}y)=(Y_{\bbA}y)(x)=\bbA(x,y)
\end{equation}
for all $x,y\in {\bbA}_{0}$. That is to say, the Yoneda embedding is \emph{fully faithful}.}
\isarnote{Moved a remark into a footnote higher up.}

Each $\Q$-functor $F: {\bbA}\to {\bbB}$ \isar{determines a} pair of $\Q$-functors,
\begin{equation}\label{AI}
F^{\to}:{\P}{\bbA}\to{\P}{\bbB}:\isar{\mu\mapsto \mu\circ F^{\natural}}\quad\mbox{ and }\quad F^{\leftarrow}: {\P}{\bbB}\to{\P}{\bbA}:\isar{\lambda\mapsto \lambda\circ F_{\natural}}
\end{equation}
\isar{such that} $F^{\to}\dashv F^{\leftarrow}$ in $\Cat(\Q)$. For $F: {\bbA}\to {\bbB}$ and $G: {\bbB}\to \mathbb{C}$, we have
\begin{equation}\label{AJ}
(G\circ F)^{\to}=G^{\to}\circ F^{\to}\quad\mbox{ and }\quad(G\circ F)^{\leftarrow}=F^{\leftarrow}\circ G^{\leftarrow}.
\end{equation}
\begin{definition}\label{AK}\isarnote{Made this a definition. Put "preservation of colim" in the definition for ease of reference later on.}
Given a $\Q$-functor $F:{\bbB}\to {\bbA}$ and a ${\Q}$-distributor $\theta: {\bbB}\dist\mathbb{C}$, the \emph{colimit} of $F$ weighted by $\theta$ is \isar{(whenever it exists)} a $\Q$-functor ${\colim}(\theta, F): \mathbb{C}\to {\bbA}$ such that
\begin{equation}\label{AL}
\isar{{\colim}(\theta, F)_{\natural}}=F_{\natural}\swarrow\theta.
\end{equation}
\isar{That colimit is \emph{preserved} by a functor $G\:\bbA\to\bbA'$ if $\colim(\theta,G\circ F)$ (exists and) equals $G\circ\colim(\theta,F)$.}
\end{definition}
\isar{Colimits need not be unique, but are always essentially unique (up to isomorphism).}

For a presheaf $\mu: {\bbA}\dist \{q\}$ and a $\Q$-functor $F: {\bbA}\to {\bbB}$, ${\colim}(\mu, F)$ is in principle a $\Q$-functor from $\{q\}$ to ${\bbB}$. But such a $\Q$-functor simply picks out an object of type $q$ in ${\bbB}$. Therefore we will often think of ${\colim}(\mu, F)$ just as being that object \cite{Stubbe20051}. \isarnote{Moved this part.}\isar{In particular, given $\mu\in ({\P}{\bbA})_{0}$, if $\colim(\mu,1_{\bbA})$ exists in $\bbA$, then we call it ``the'' \emph{supremum} of $\mu$, and write it as $\sup(\mu)$. We may thus consider $\sup(\mu)$ as an object of $\bbA$ of type $t\mu$, determined up to isomorphism, such that}
\begin{equation}\label{AM}
{\bbA}(\sup(\mu),-)=\isar{(1_{\bbA})_{\natural}\swarrow\mu}={\bbA}\swarrow\mu.
\end{equation}

Since $F_{\natural}\swarrow \mu=(1_{\bbA}\circ F)_{\natural}\swarrow \mu=((1_{\bbA})_{\natural}\circ F_{\natural})\swarrow \mu=(1_{\bbA})_{\natural}\swarrow F^{\natural}\swarrow \mu$, the colimit of a $\Q$-functor $F: {\bbA}\to {\bbB}$ weighted by a presheaf $\mu\in ({\P}{\bbA})_{0}$ is the supremum of $\mu\circ F^{\natural}$. Furthermore, it is routine to check that the colimit of $F:{\bbB}\to {\bbA}$ weighted by a ${\Q}$-distributor $\theta: {\bbB}\dist\mathbb{C}$ is a $\Q$-functor sending each object $c\in \mathbb{C}_{0}$ to the \isar{supremum} of $\theta(-, c)\circ F^{\natural}$, i.e., ${\colim}(\theta, F)(c)=\sup (\theta(-, c)\circ F^{\natural}).$ Thus, a ${\Q}$-category ${\bbA}$ admits all weighted colimits if and only if every $\mu\in({\P}{\bbA})_{0}$ has a supremum.
\begin{definition}[\cite{Stubbe20051,laishen}]\label{AN} A ${\Q}$-category ${\bbA}$ is said to be \emph{cocomplete} if it satisfies one of the following equivalent conditions:
\begin{enumerate}[label=(\arabic*),nosep]
\item ${\bbA}$ admits all weighted colimits;
\item every $\mu\in({\P}{\bbA})_{0}$ has a supremum;
\item the Yoneda embedding $Y_{\bbA}: {\bbA}\to {\P}{\bbA}$ has a left adjoint
 $\sup: {\P}{\bbA}\to {\bbA}$ in $\Cat(\Q)$.
\end{enumerate}
\end{definition}
Completeness of ${\Q}$-categories is defined dually. A ${\Q}$-category is complete if and only if it is cocomplete, see \cite[Proposition 5.10.]{Stubbe20051}. The category ${\P}{\bbA}$ of presheaves on a ${\Q}$-category ${\bbA}$ is always cocomplete \cite[Proposition 6.4.]{Stubbe20051}: for $\Psi\in ({\P}({\P}{\bbA}))_{0}$ \isar{one can check that}
\begin{equation}\label{AO}
\sup(\Psi)=Y_{\bbA}^{\leftarrow}(\Psi)=\Psi\circ (Y_{\bbA})_{\natural}.
\end{equation}
\isar{Note that this is an equality (and not a mere isomorphism), because $\P\bbA$ is skeletal.}
\begin{example}\label{AP}\isarnote{Made this "example" instead of "remark".} (1) A unital quantale \isar{$(Q,\bigvee,\cdot,1)$} is \isar{precisely (the hom-lattice of)} a quantaloid with \isar{a single} object. If ${\Q}$ is \isar{(the one-object suspension of) }a unital commutative quantale, then ${\Q}$-categories are precisely quantale-enriched categories studied in \cite{LaiZhang2007,Wagner1997}. Since a complete residuated lattice $L$ is a commutative integral quantale, we see that $L$-ordered sets in the sense of B\v{e}lohl\'{a}vek and Fan \cite{Bel2002,Fan2001,Yao2010} are special $\Q$-categories. Thus, ${\Q}$-categories can be interpreted as a generalization of quantale-enriched categories in a many-typed setting.

(2) Let {\bf 2} be the two-element Boolean algebra. Then {\bf 2}-categories are \isar{preordered sets}, {\bf 2}-distributors are ideal relations, {\bf 2}-functors are order-preserving maps and presheaves are lower sets. Thus, ${\Q}$-categories can be interpreted as a generalization of \isar{(pre)ordered sets} in a many-valued and many-typed setting. If ${\bbA}$ is a {\bf 2}-category, then the category ${\P}{\bbA}$ is just the poset of lower sets on ${\bbA}$ with inclusion order.

(3) Let $[0,\infty]_{+} = ([0,\infty],\isar{\bigwedge},+, 0)$ be the quantale whose underlying complete lattice is the extended non-negative real line $[0, \infty]$ equipped with the order $\geq$. A \emph{partial metric} \cite{Matthews1992,Waszkiewicz2003} on a set $X$ is a map $p: X\times X \to [0,\infty)$ which satisfies for all $x, y, z\in X$,
\begin{enumerate}[label=(\roman*),nosep]
\item $p(x, y) = p(y, x)$ (symmetry).
\item $p(x, y) = p(x, x) = p(y, y)$ implies $x = y$ ($\mathrm{T}_{0}$ separation axiom).
\item $p(x, y)\leq p(x, z) - p(z, z) + p(z, y)$ (the sharp triangle inequality).
\item $p(x, x)\leq p(x, y)$ (SSD -- ``small self-distances'')
\end{enumerate}
If we abandon axioms (i) and (ii) and extend the range of $p$ to $[0, +\infty]$, then $p$ is called a \emph{generalized partial \isar{metric}} \cite{PuZhang2012}. Let $\textsf{D}[0,\infty]_{+}$ denote the quantaloid of \isar{\emph{diagonals} in} $[0,\infty]_{+}$. It is remarkable that (generalized) partial metric spaces are exactly $\textsf{D}[0, \infty]_{+}$-categories, \isar{see \cite{Hohle2011,PuZhang2012}, and especially \cite[Example 2.14]{Stubbe2014} and \cite{HS2018}.}

(4) \isar{For any object $A$ in a quantaloid $\Q$, we shall write $\Q^A$ for the $\Q$-category $\P\{A\}$ of contravariant presheaves on the one-object $\Q$-category $\{A\}$. As in \cite[Example 3.9]{Stubbe20051}, we can explicitly describe $\Q^A$: the set $({\Q}^{A})_{0}$ contains all ${\Q}$-arrows with domain $A$; the type of $f:A\to X$ is its codomain $X$; and for $f: A\to X$ and $g:A\to Y$, we have $({\Q}^{A})(g, f)=f\swarrow g$ in $\Q(Y,X)$. As any category of presheaves, $\Q^A$ is (complete and) cocomplete.} Specifically, for $\mu\in({\P}{\Q}^{A})_{0}$, \isar{it follows from (2.6)} that $\sup(\mu)=\mu(1_{A})$ \isar{since the Yoneda embedding $Y_{\{A\}}:\{A\}\to\P\{A\}=\Q^A$ maps the single object of $\{A\}$ to $1_A$.} \isar{Note} that, if ${\Q}$ has more than one element, then ${\Q}^{A}$ is a ${\Q}$-category, which can not fall into the framework of quantale-enriched categories.
\end{example}
\isar{Many more examples can be found in the references}.\isarnote{Removed a remark.}

\section{$\Phi$-cocomplete ${\Q}$-categories}\label{S3}

In order to develop domain theory in ${\Q}$-categories, the first step is to build the concept of directed completeness \cite{Albert,Kelly2005}. Based on the successful experiences of studying domain theory through subset systems, and \isar{taking into account} the fact that there is no standard concept of direct\isar{ed}ness in ${\Q}$-categories, we shall \isar{work abstractly with ${\Q}$-categories which are cocomplete with respect to a suitable class of weights}.

\begin{definition}\label{class}\isarnote{Made this a definition}
\isar{A \emph{class of presheaves} $\Phi$ is determined by} a set $(\Phi{\bbA})_{0}\subseteq(\P{\bbA})_0$ for each ${\Q}$-category ${\bbA}$, such that
\begin{enumerate}[nosep,label=(\roman*)]
\item each $(\Phi{\bbA})_{0}$ contains all representable presheaves on $\bbA$,
\item for every $\mu\in(\Phi{\bbA})_{0}$ and every $\Q$-functor $F: {\bbA}\to {\bbB}$,  $\mu\circ F^{\natural}\in (\Phi{\bbB})_{0}$.
\end{enumerate}
We call $\mu\in(\Phi{\bbA})_{0}$ a \emph{$\Phi$-ideal} in ${\bbA}$.
\end{definition}
\isar{Obviously, the $\Phi$-ideals in $\bbA$ are the objects of a full $\Q$-subcategory $\Phi\bbA$ of $\P\bbA$; we shall write the inclusion functor as $i_{\bbA}:\Phi\bbA\to\P\bbA$. As $\P\bbA$ is skeletal, so is $\Phi\bbA$. The first condition in the above definition says exactly that the Yoneda embedding $Y_{\bbA}:\bbA\to\P\bbA$ co-restricts to a $\Q$-functor $Y_{\bbA}:\bbA\to\Phi\bbA$; the second condition in the above definition says exactly that, for any $\Q$-functor $F:\bbA\to\bbB$, the $\Q$-functor $F^{\to}:{\P}{\bbA}\to{\P}{\bbB}$ (co)restricts to a $\Q$-functor $F^{\to}:{\Phi}{\bbA}\to{\Phi}{\bbB}$.}

From now on, we use $\Phi$ to denote a class of presheaves.
\begin{definition}\label{def-phi-cocomplete}
A ${\Q}$-category ${\bbA}$ is said to be \emph{$\Phi$-cocomplete} if \isar{every $\Phi$-ideal in $\bbA$ has a supremum}, that is, for every $\mu\in(\Phi{\bbA})_{0}$, $\sup(\mu)=\colim(\mu,1_\bbA)$ exists.
\end{definition}
An element $a$ of type $q$ in a ${\Q}$-category ${\bbA}$ \isar{determines, and is determined by,} a ${\Q}$-functor $a:\{q\}\to {\bbA}:q\mapsto a$. Let $\phi: {\bbA}\dist{\bbB}$ be a ${\Q}$-distributor, $b\in{\bbB}_{0}$. Then $\phi(-, b)=b^{\natural}\circ\phi$ is a presheaf on ${\bbA}$. A ${\Q}$-distributor $\phi: {\bbA}\dist{\bbB}$ is said to be a \emph{$\Phi$-distributor}, if for every $b\in {\bbB}_{0},$ $\phi(-, b)$ is a $\Phi$-ideal in ${\bbA}.$ Clearly, a $\Phi$-ideal is a particular $\Phi$-distributor. By the relation between \isar{suprema} of presheaves and \isar{general} weighted colimits, it follows that:
\begin{proposition}\label{Phi-cocomplete}\isarnote{Made it a proposition.}
\isar{For any ${\Q}$-category $\bbA$, the following are equivalent:}
\begin{enumerate}[label=(\arabic*),nosep]
\item ${\bbA}$ is $\Phi$-cocomplete,
\item the co-restricted Yoneda embedding $Y_{\bbA}\:{\bbA}\to\Phi{\bbA}$ is right adjoint \isar{(to $\mu\mapsto\sup(\mu)$)},
\item ${\bbA}$ has all colimits weighted by $\Phi$-ideals,\isarnote{Made shorter.}
\item ${\bbA}$ has all colimits weighted by $\Phi$-distributors.
\end{enumerate}
\end{proposition}
\begin{definition}\label{cocontinuous}
A $\Q$-functor $F: {\bbA}\to {\bbB}$ between two ${\Q}$-categories is said to be \emph{$\Phi$-cocontinuous} if it preserves all \isar{suprema} of $\Phi$-ideals that exist in ${\bbA}$, that is, for every $\phi\in (\Phi{\bbA})_{0}$, if $\sup(\phi)$ exists, then $\sup(F^{\to}(\phi))$ exists and equals $F(\sup(\phi))$.
\end{definition}
\isar{Again using the relation between suprema of presheaves and general weighted colimits, it follows that}:
\begin{proposition}\label{prop-cocont}\isarnote{Made it a proposition, slightly shorter formulation.}
\isar{For any $\Q$-functor $F: {\bbA}\to {\bbB}$, the following are equivalent:}
\begin{enumerate}[label=(\arabic*),nosep]
\item $F$ is $\Phi$-cocontinuous,
\item $F$ preserves all colimits weighted by $\Phi$-ideals \isar{that} exist in ${\bbA}$,
\item $F$ preserves all colimits weighted by $\Phi$-distributors \isar{that} exist in ${\bbA}$.
\end{enumerate}
\end{proposition}
\isar{Since} the category ${\P}{\bbA}$ of presheaves on ${\bbA}$ is cocomplete\isar{, it certainly} is $\Phi$-cocomplete for any class of presheaves $\Phi$.\isarnote{Moved this sentence.} \isar{However, note} that $\Phi{\bbA}$ is not $\Phi$-cocomplete in general. In order to \isar{``fix'' this, we need a further condition on the class $\Phi$, that we recall first.}
\begin{definition}[\cite{Stubbe2010}]\label{saturated} A class of presheaves $\Phi$ is \emph{saturated} if for each $\mu\in (\Phi{\bbA})_{0}$ and each $G: {\bbA}\to{\P}{\bbB}$ for which each $G(a)\in (\Phi{\bbB})_{0}$, ${\colim}(\mu, G)$ is in $(\Phi{\bbB})_{0}$ too.
\end{definition}
\begin{remark}\label{saturatedbis}
\isar{If a class of presheaves $\Phi$ is saturated, then condition (ii) in Definition \ref{class} follows from condition (i) in Definition \ref{class} and the colimit-condition in Definition \ref{saturated}; and that is how it was put in \cite[Definition 4.1]{Stubbe2010}, where ``non-saturated'' classes were simply not considered. The interesting point is, however, that the saturatedness of a class of presheaves is equivalent to each $\Phi\bbA$ being $\Phi$-cocomplete, as we shall show next.}
\end{remark}
\begin{proposition}\label{prop-saturated}\isarnote{Simplified statement and proof.}
\isar{For a class of presheaves $\Phi$,} the following are equivalent:
\begin{enumerate}[label=(\arabic*),nosep]
\item \isar{for all $\bbA$,} $\Phi{\bbA}$ is $\Phi$-cocomplete and the inclusion $i_\isar{{\bbA}}:\Phi{\bbA}\to {\P}{\bbA}$ is $\Phi$-cocontinuous,
\item $\Phi$ is saturated,
\item the composite of any two $\Phi$-distributors is a $\Phi$-distributor.
\end{enumerate}
In this case, for all $\Psi\in (\Phi(\Phi{\bbA}))_{0}$, $\sup \Psi=\sup(i_{\bbA}^{\to}(\Psi))$.
\end{proposition}
\begin{proof}
(1$\Rightarrow$2) \isar{Suppose that $\mu:{\bbA}\dist\{q\}$ is in $(\Phi\bbA)_0$ and that $G:{\bbA}\to{\P}{\bbB}$ is a $\Q$-functor for which each $G(a)$ is in $(\Phi\bbA)_0$. Then $G$ is the composite of its co-restriction $G':\bbA\to\Phi\bbB$ and the inclusion $i_{\bbB}:\Phi\bbB\to\P\bbB$. By assumption, $\colim(\mu,G')$ exists in $\Phi\bbB$ and is preserved by $i_{\bbB}$, that is, $i_{\bbB}(\colim(\mu,G'))=\colim(\mu,i_{\bbB}\circ G')=\colim(\mu,G)$. Therefore the latter is in $(\Phi\bbB)_0$, as wanted.}\\
(2$\Rightarrow$1) \isar{This is in \cite[Proposition 4.4]{Stubbe2010}.}\\
($2\Leftrightarrow 3$) \isar{This is in \cite[Proposition 4.2]{Stubbe2010}.}\\
\isar{Assuming now the (equivalent) conditions in the statement, for any $\Psi\in(\Phi(\Phi\bbA))_0$ both $\colim(\Psi,1_{\Phi\bbA})=\sup(\Psi)$ and $\colim(\Psi,i_{\bbA})=\colim(i_{\bbA}^{\to}(\Psi),1_{\P\bbA})=\sup(i_{\bbA}^{\to}(\Psi))$ exist, and $\colim(\Psi,1_{\Phi\bbA})=i_{\bbA}(\colim(\Psi,1_{\Phi\bbA}))=\colim(\Psi,i_{\bbA})$. This explains the stated formula.}
\end{proof}
\begin{remark}\label{BB}
(1) If ${\Q}$ is a commutative unital quantale, then the concept of a saturated class of presheaves agrees with that of \emph{a saturated class of weights} in quantale-enriched categories given in \cite{LaiZhang2007}. Thus, a saturated class of presheaves on $\Cat(\Q)$ can be seen as extension of a saturated class of weights on quantale-enriched categories in a many-typed setting.

(2) \isarnote{Earlier in the paper I removed the restriction to "skeletal" cats; but here I added a remark to the effect that preorder/partial order makes no big difference.}Following the pioneer work of Wright, Wagner and Thatcher \cite{Wright1978} on so-called subset systems $Z$, order-theoretical, topological and categorical properties of posets have been extended to the general $Z$-setting \cite{Bandelt1983,Erne,LiuZhao2013}. Recall that {\bf 2} is the two-element Boolean algebra, $\Cat({\bf 2})$ is the category \isar{of preordered sets} and order-preserving maps. Thus, if $\Phi$ is a saturated class of presheaves on $\Cat({\bf 2})$, then it is a union-complete subset system in the sense of \cite{Wright1978}, \isar{albeit on preorders and with all $\Phi$-ideals being lower sets. Strictly speaking, union-complete subset systems $Z$ were only defined for posets, and $Z$-sets need not be a lower sets. However, this has} little or no impact on developing the theory of completeness and continuity of posets/preorders in a general setting. Thus, the concept of a saturated class of presheaves on $\Cat(\Q)$ can be viewed as a generalization of that of a union-complete subset system on {\bf Poset} in a many-typed and many-valued setting.

(3)\isarnote{Reformulated and shortened the end of this section, and put it in the Remark.}
\isar{In \cite{Stubbe2010}, Stubbe studied more categorical properties of $\Phi$-cocomplete ${\Q}$-categories and $\Phi$-cocontinuous $\Q$-functors. In particular, \cite[Theorem 4.6]{Stubbe2010}, contains the fact that there is an essentially bijective correspondence between saturated classes $\Phi$ of presheaves on ${\Q}$-categories on the one had, and so-called \emph{full sub-KZ-doctrines} $(\mathcal{T}, \varepsilon)$ of the free cocompletion KZ-doctrine ${\P}:\Cat(\Q)\to\Cat(\Q)$ on the other. Particularly, any saturated class of presheaves $\Phi$ determines the KZ-doctrine (a special 2-monad)
\begin{equation}\label{BC}
\Phi: \Cat(\Q)\to \Cat(\Q): (F: {\bbA}\to {\bbB})\mapsto (F^{\to}: \Phi{\bbA}\to \Phi{\bbB})
\end{equation}
whose image is precisely the category of $\Phi$-cocomplete $\Q$-categories and $\Phi$-cocontinuous functors. We refer to \cite{Stubbe2010} for details.}

(4) \isarnote{Added this in the remark.} \isar{The minimal class of
presheaves $\Phi_{min}$ is that which contains only the
representable presheaves; it is saturated, and
$\Phi_{min}\bbA=Y_{\bbA}(\bbA)$. The maximal class is that which
contains all presheaves; it is also saturated, and
$\Phi_{max}\bbA=\P\bbA$. In the next section we shall study more
interesting examples.}
\end{remark}

\section{Examples of saturated classes of presheaves}\label{S4}

With the rapid development of quantitative domain theory, ideals in quantale-enriched categories and in $L$-ordered sets are studied and there are some elegant results \cite{LaiZhang2007,LiuZhao2014,Laizhang2020,Yao2010,ZhangFan2005}. Recently, Lai and Zhang gave a comparative study of three kinds of ideals in fuzzy ordered sets \cite{Laizhang2020}, \isar{each of generalizing well-known posetal definitions.}

\isar{Let us first recall the posetal situation. For a non-empty lower set $I$ of a poset $P$,} the following are equivalent:
\begin{enumerate}[label=(\roman*),nosep]

\item $I$ is an \emph{ideal}\isar{:} for every $x, y\in I$, there is some $z\in I$ such that $x, y\leq z$;

\item $I$ is \emph{irreducible}\isar{:} for lower sets $B,C$ of $P$, $I\subseteq B\cup C$ implies that $I\subseteq B$ or $I\subseteq C$;

\item $I$ is \emph{flat}\isar{:} for upper sets $G,H$ of $P$, if $I$ intersects with $G$ and \isar{with} $H$, then $I$ intersects with $G\cap H$.
\end{enumerate}
\isar{The above characterizations of ideals in posets have been generalized to quantale-enriched categories (or $L$-ordered sets) in \cite{LaiZhang2007,Laizhang2020,LiuZhao2014,Yao2010}, where they are no longer equivalent, as pointed out in $\cite{Laizhang2020}$. In what follows, we shall introduce saturated classes of presheaves that further generalize the above definitions to the `many-valued, many-typed' setting of quantaloid-enriched categories.}

Directed sets and ideals in domain theory are usually assumed to be non-empty $\cite{Gie2003}$. Based \isar{on the} definition of inhabited fuzzy sets in \cite{Laizhang2020}, we shall give \isar{a} corresponding concept \isar{for} ideals in ${\Q}$-categories.

\begin{definition}\label{inhabited}\isarnote{Made a definition and reformulated slightly.}
A presheaf $\mu:\bbA\dist\{q\}$ is \emph{inhabited} if
\begin{equation}\label{DA}
1_{q}\leq \bigvee_{a\in {\bbA}_{0}, ta=q}\mu(a)
\end{equation}
A copresheaf $\lambda:\{p\}\dist\bbA$ is \emph{inhabited} if
\begin{equation}\label{DB}
 1_{p}\leq \bigvee_{a\in {\bbA}_{0}, ta=p}\lambda(a).
\end{equation}
\end{definition}
\isar{Next we recall that $\Dist(\Q)$ is a quantaloid, so we can compute suprema (and hence also infima) of parallel distributors. This applies in particular to parallel (co- or contravariant) presheaves, and we use this in the statements below.}
\begin{definition}\label{ideals}\isarnote{Made a single definition and reformulated slightly}
A presheaf $\phi\in({\P}{\bbA})_{0}$ is:
\begin{enumerate}[nosep]
\item an \emph{irreducible ideal} if
\begin{equation}\label{DC}
{\P}{\bbA}(\phi, \phi_{1}\vee \phi_{2})={\P}{\bbA}(\phi, \phi_{1})\vee{\P}{\bbA}(\phi, \phi_{2})
\end{equation}
holds for all $\phi_{1}, \phi_{2}\in ({\P}{\bbA})_{0}$ with the same type;
\item a \emph{flat ideal} if
\begin{equation}\label{DD}
\phi\circ (\lambda_{1}\wedge \lambda_{2})=\phi\circ \lambda_{1}\wedge\phi\circ \lambda_{2}
\end{equation}
holds for all $\lambda_{1}, \lambda_{2}\in ({\P}^{\dag}{\bbA})_{0}$ with the same type;
\item a \emph{weakly flat ideal} if
\begin{equation}\label{DE}
\phi\circ (\lambda_{1}\wedge \lambda_{2})=\phi\circ \lambda_{1}\wedge\phi\circ \lambda_{2}
\end{equation}
holds for all inhabited copresheaves $\lambda_{1}, \lambda_{2}\in ({\P}^{\dag}{\bbA})_{0}$ with the same type.
\end{enumerate}
\end{definition}

\begin{remark}\label{CA}
(1) In \cite{Tao2014}, Tao, Lai and Zhang investigated flat weights in quantaloid-enriched categories. Later, Lai studied irreducible ideals and flat ideals in quantale-enriched categories in \cite{Laizhang2020}. Note that, in \cite{Laizhang2020}, irreducible ideals and flat ideals are supposed to be inhabited. Flat ideals in the sense of this paper are called weakly flat weights in \cite{Tao2014}.

(2) Following the definition of an ideal given in \cite{LaiZhang2007}, order-theoritical and categorical properties of quantale-enriched categories are studied in a series of papers \cite{LiuZhao2014,Yao2010}. By Propositions 3.9 and 3.11 in \cite{Tao2014}, we know that if ${\Q}$ is a frame then $\phi$ is an inhabited flat ideal if and only if $\phi$ is an ideal in the sense of \cite{LaiZhang2007}.

(3) More particularly, suppose that ${\Q}$ is the two-element Boolean algebra {\bf 2} \isar{and that} ${\bbA}$ is a \isar{skeletal} {\bf 2}-category. By the fact that ${\P}{\bbA}$ is the set of all lower sets of ${\bbA}$ with inclusion order, we know that the following are equivalent:
\begin{enumerate}[label=(\roman*),nosep]

\item $\phi$ is an ideal in the poset ${\bbA}_{0}$;

\item $\phi$ is an inhabited irreducible ideal in the sense of \isar{Definitions \ref{inhabited} and \ref{ideals}};

\item $\phi$ is an inhabited flat ideal in the sense of \isar{Definitions \ref{inhabited} and \ref{ideals}}.
\end{enumerate}
\end{remark}
\begin{theorem}\label{CB}
\isar{Let $\mathcal{I}$ (resp., $\mathcal{I}_{in}$) be the class of all irreducible ideals (resp., inhabited irreducible ideals). These are saturated classes of presheaves.}
\end{theorem}
\begin{proof}
\isar{Clearly, both classes contain all representable presheaves. Next,} let ${\phi: {\bbA}\dist \{q\}}$ be an irreducible ideal and $G: {\bbA}\to{\P}{\bbB}$ a $\Q$-functor such that each $G(a)$ is irreducible. For $\phi_{1}, \phi_{2}\in ({\P}{\bbB})_{0}$ with same type, we have \begin{align} &{\P}{\bbB}({\colim}(\phi, G), \phi_{1}\vee\phi_{2})\\ &={\P}{\bbB}(\sup(\phi\circ G^{\natural}), \phi_{1}\vee\phi_{2})\\ &={\P}{\P}{\bbB}(\phi\circ G^{\natural}, Y_{{\P}{\bbB}}(\phi_{1}\vee\phi_{2}))&(\sup\dashv Y_{{\P}{\bbB}})\\ &={\P}{\bbA}(\phi, Y_{{\P}{\bbB}}(\phi_{1}\vee\phi_{2})\circ G_{\natural})&(G^{\to}\dashv G^{\leftarrow})\\ &={\P}{\bbA}(\phi, {\P}{\bbB}(G-, \phi_{1}\vee\phi_{2}))\\ &={\P}{\bbA}(\phi, {\P}{\bbB}(G-, \phi_{1})\vee{\P}{\bbB}(G-, \phi_{2}))&(\mathrm{each}\ G(a)\ \mathrm{is\ irreducible})\\ &={\P}{\bbA}(\phi, {\P}{\bbB}(G-, \phi_{1}))\vee{\P}{\bbA}(\phi, {\P}{\bbB}(G-, \phi_{2}))&(\phi\ \mathrm{is\ irreducible})\\ &={\P}{\bbA}(\phi, Y_{{\P}{\bbB}}(\phi_{1})\circ G_{\natural})\vee{\P}{\bbA}(\phi, Y_{{\P}{\bbB}}(\phi_{2})\circ G_{\natural})\\ &={\P}{\P}{\bbB}(\phi\circ G^{\natural}, Y_{{\P}{\bbB}}(\phi_{1}))\vee{\P}{\P}{\bbB}(\phi\circ G^{\natural}, Y_{{\P}{\bbB}}(\phi_{2}))\\ &={\P}{\bbB}({\colim}(\phi, G), \phi_{1})\vee {\P}{\bbB}({\colim}(\phi, G), \phi_{2}). \end{align} This suffices to conclude that $\mathcal{I}$ is a saturated class of presheaves (see also Remark \ref{saturatedbis}). In addition, suppose that $\phi$ and each $G(a)$ are inhabited. We have \begin{align} \bigvee_{b\in{\bbB}_{0},tb=q}{\colim}(\phi, G)(b)&=\bigvee_{b\in{\bbB}_{0},tb=q}(\phi\circ G^{\natural}\circ (Y_{\bbB})_{\natural})(b)\\ &=\bigvee_{b\in{\bbB}_{0},tb=q}\bigvee_{a\in{\bbA}_{0}}\bigvee_{\psi\in({\P}{\bbB})_{0}}\phi(a)\circ{\P}{\bbB}(\psi, G(a))\circ {\P}{\bbB}(Y_{\bbB}(b),\psi)\\ &\geq\bigvee_{b\in{\bbB}_{0},tb=q}\bigvee_{a\in{\bbA}_{0}}\phi(a)\circ{\P}{\bbB}(Y_{\bbB}(b), G(a))\\ &\geq\bigvee_{a\in{\bbA}_{0}, ta=q}\left(\phi(a)\circ\bigvee_{b\in{\bbB}_{0},tb=q}G(a)(b)\right)\\ &\geq\bigvee_{a\in{\bbA}_{0}, ta=q}(\phi(a)\circ 1_{q})\\ &=\bigvee_{a\in{\bbA}_{0}, ta=q}\phi(a)\\ &\geq1_{q}. \end{align} Thus, $\mathcal{I}_{in}$ is a saturated class of presheaves.
\end{proof}

\begin{theorem}\label{CC}
\isar{Let $\mathcal{F}$ (resp., $\mathcal{F}_{in}$, $\mathcal{WF}$) be the class of flat ideals (resp., inhabited flat ideals, weakly flat ideals). These are saturated classes of presheaves.}
\end{theorem}
\begin{proof}
\isar{Clearly, these classes contain all representable presheaves. Next, }let ${\phi: {\bbA}\dist \{q\}}$ be a flat ideal and $G: {\bbA}\to{\P}{\bbB}$ a $\Q$-functor such that each $G(a)$ is flat. For $\lambda_{1}, \lambda_{2}\in ({\P}^{\dag}{\bbB})_{0}$ with same type, we have \begin{align} {\colim}(\phi, G)\circ (\lambda_{1}\wedge\lambda_{2})&=\phi\circ G^{\natural}\circ (Y_{\bbB})_{\natural}\circ(\lambda_{1}\wedge\lambda_{2})\\ &=\phi\circ{\P}{\bbB}(Y_{\bbB}-, G-)\circ(\lambda_{1}\wedge\lambda_{2})\\ &=\phi\circ({\P}{\bbB}(Y_{\bbB}-, G-)\circ\lambda_{1}\wedge {\bbA}(Y_{\bbB}-, G-)\circ\lambda_{2})\\ &=\phi\circ{\P}{\bbB}(Y_{\bbB}-, G-)\circ\lambda_{1}\wedge \phi\circ{\bbA}(Y_{\bbB}-, G-)\circ\lambda_{2}\\ &=\phi\circ G^{\natural}\circ (Y_{\bbB})_{\natural}\circ\lambda_{1}\wedge \phi\circ G^{\natural}\circ (Y_{\bbB})_{\natural}\circ\lambda_{2}\\ &={\colim}(\phi, G)\circ \lambda_{1}\wedge {\colim}(\phi, G)\circ \lambda_{2}. \end{align} So, we can conclude that $\mathcal{F}$ and $\mathcal{WF}$ are saturated classes of presheaves. By similar arguments given in Theorem 4.4, we obtain that $\mathcal{F}_{in}$ is a saturated class of presheaves.
\end{proof}

\isar{Finally,} based on the class of \emph{conical presheaves} studied in \cite[Section 5]{Stubbe2010}, we shall introduce the third kind of ideals in ${\Q}$-enriched categories. \isar{Below we shall write $(\bbA_q,\leq_q)$ for the preordered set of objects of type $q$ in ${\bbA}$.}
\begin{definition}\label{CD}
A presheaf ${\phi: {\bbA}\dist \{q\}}$ is a
\begin{enumerate}[nosep,leftmargin=*,label={-}]
\item \emph{conical presheaf} if there exists a set $\{a_{i}\}_{i\in I}$ in $({\bbA}_{q}, \leq_{q})$ such that $\phi=\bigvee_{i\in I}{\bbA}(-, a_{i})$;
\item \emph{conical ideal} if there exists a directed set $\{a_{i}\}_{i\in I}$ in $({\bbA}_{q}, \leq_{q})$ such that $\phi=\bigvee_{i\in I}{\bbA}(-, a_{i})$.
\end{enumerate}
\end{definition}
\isar{Since every directed set is nonempty, every conical ideal is inhabited.}
\begin{theorem}\label{CE}
Let $\mathcal{C}$ (resp.\ $\mathcal{C}^{D}$) be the class of conical presheaves (resp.\ conical ideals). These are saturated class of presheaves.
\end{theorem}
\begin{proof}
\isar{Clearly, both these classes contain all representable presheaves. That $\mathcal{C}$ is a saturated class of presheaves is proved in \cite[Proposition 5.3]{Stubbe2010}. Now let} ${\phi: {\bbA}\dist \{q\}}$ be a conical ideal and $G: {\bbA}\to{\P}{\bbB}$ a $\Q$-functor such that each $G(a)$ is a conical ideal. Then there exist is a directed set $\{a_{i}\}_{i\in I}$ in $({\bbA}_{q}, \leq_{q})$ such that $\phi=\bigvee_{i\in I}{\bbA}(-, a_{i})$. For each $a_{i}$, there is a directed set $\{b^{i}_{k}\}_{k\in J_{i}}$ such that $G(a_{i})=\bigvee_{k\in J_{i}}{\bbB}(-, b^{i}_{k})$. \begin{align} {\colim}(\phi, G)&=\phi\circ G^{\natural}\circ (Y_{\bbB})_{\natural}\\ &=\phi\circ{\P}{\bbB}(Y_{\bbB}-, G-)\\ &=\left(\bigvee_{i\in I}{\bbA}(-, a_{i})\right)\circ{\P}{\bbB}(Y_{\bbB}-, G-)\\ &=\bigvee_{i\in I}\left({\bbA}(-, a_{i})\circ{\P}{\bbB}(Y_{\bbB}-, G-)\right)\\ &=\bigvee_{i\in I}{\P}{\bbB}(Y_{\bbB}-, G(a_{i}))\\ &=\bigvee_{i\in I}G(a_{i})\\ &=\bigvee_{i\in I}\bigvee_{k\in J_{i}}{\bbB}(-, b^{i}_{k}) \end{align} By the directedness of $\{a_{i}\}_{i\in I}$ and $\{b^{i}_{k}\}_{k\in J_{i}}$, also $\{b^{i}_{k}\mid i\in I, k\in J_{i}\}$ is a directed set. Thus ${\colim}(\phi, G)$ is a conical ideal.
\end{proof}

\section{Continuity in ${\Q}$-categories}\label{S5}

\isar{As mentioned before, Stubbe \cite{Stubbe2007} studied ``totally continuous'' quantaloid-enriched categories, i.e.\ \emph{cocomplete} quantaloid-enriched categories which satisfy a further continuity condition with respect to the class of \emph{all presheaves}. For quantale-enriched categories, Hofmann and Waszkiewicz \cite{HW2011} expressed continuity with respect to any given class of presheaves, without any cocompleteness requirement. In this section, we shall combine these works to appropriately formulate $\Phi$-continuity for quantaloid-enriched categories.}

\isarnote{Reorganized and reformulated.}\isar{
Recall that, given a class of presheaves $\Phi$ on $\Cat(\Q)$, $\Phi\bbA$ is the category of $\Phi$-ideals in $\bbA$. It is a full subcategory of $\P\bbA$, and we know that the Yoneda embedding $Y_{\bbA}\:\bbA\to\P\bbA$ co-restricts to $\Phi\bbA\subseteq\P\bbA$. Now consider a further full subcategory of $\Phi\bbA$ (and thus of $\P\bbA$), determined by
\begin{equation}\label{EA}
(\Phi_{s}{\bbA})_0=\{\phi\in (\Phi{\bbA})_{0}\mid\sup(\phi)\mbox{ exists}\}.
\end{equation}
Because $\sup(Y_{\bbA}(a))=a$ for any $a\in\bbA_0$, and $Y_{\bbA}(\sup(\mu))\geq\mu$ for any $\mu\in\P\bbA$ whenever the supremum exists, we find that the co-restriction of the Yoneda embedding is right adjoint to the restriction of the supremum map:}
\begin{equation}\label{adj}
\bbA\xymatrix@C=8ex{\ar@/_2ex/[r]_{Y_{\bbA}}\ar@{}[r]|{\perp} & \ar@/_2ex/[l]_{\sup}}\Phi_s\bbA
\end{equation}
\begin{definition}\label{way-below}
For a ${\Q}$-category ${\bbA}$ and a class of presheaves $\Phi$, the \emph{$\Phi$-way-below distributor} on $\bbA$ is the distributor $\Downarrow: {\bbA}\dist{\bbA}$ defined by
\[\isar{\xymatrix@C=8ex@R=2ex{\bbA\ar@{-->}[dd]|{\distsign}_{\Downarrow\ :=(\sup)^{\natural}\searrow (Y_{\bbA})_{\natural}}\ar[dr]|{\distsign}^{(Y_{\bbA})_{\natural}} \\
& \Phi_s\bbA \\
\bbA\ar[ur]|{\distsign}_{(\sup)^{\natural}}}}\]
and the corresponding \emph{way-below functor} \isar{is defined as $F_{\Downarrow}:{\bbA}\to {\P}{\bbA}:a\mapsto\ \Downarrow^{\Phi}\!(-,a)$.}
\end{definition}
\isarnote{Left out the superscript $\Phi$.}\isar{The $\Q$-distributor $\Downarrow$ and the $\Q$-functor $F_{\Downarrow}$ determine each other under the ``classifying property'' of the $\Q$-category $\P\bbA$, cf.\ \cite[Proposition 6.1]{Stubbe20051}.}
\begin{proposition}\label{prop-1}
\isar{For any ${\Q}$-category $\bbA$ we have that $\Downarrow\ \leq{\bbA}$ (and thus also $F_{\Downarrow}\leq Y_{\bbA}$ and $\Downarrow\circ\Downarrow\ \leq\ \Downarrow$)}.
\end{proposition}
\begin{proof}
\isarnote{Made argument avoiding computations and left out obvious part of proof.} Consider the following diagram of $\Q$-distributors:
\[\xymatrix@C=8ex@R=2ex{\bbA\ar[dd]|{\distsign}_{\Downarrow}\ar[dr]|{\distsign}^{(Y_{\bbA})_{\natural}} \\
& \Phi_s\bbA\ar[r]|{\distsign}^{(Y_{\bbA})^{\natural}} & \bbA \\
\bbA\ar[ur]|{\distsign}_{(\sup)^{\natural}}}\]
By definition of $\Downarrow$ we know that $(\sup)^\nat\circ\Downarrow\ \leq\ (Y_{\bbA})_\nat$, so composing with $(Y_{\bbA})^\nat$ provides
$(Y_{\bbA})^{\natural}\circ(\sup)^{\natural}\circ\Downarrow\ \leq\ (Y_{\bbA})^{\natural}\circ(Y_{\bbA})_{\natural}$.
Since $(Y_{\bbA})^{\natural}\circ(\sup)^{\natural}=(\sup\circ Y_{\bbA})^{\natural}=(1_{\bbA})^\nat=\bbA$ and $(Y_{\bbA})^{\natural}\circ(Y_{\bbA})_{\natural}=\bbA$, we find $\bbA\circ\Downarrow\ \leq\bbA$, from which $\Downarrow\ \leq\bbA\searrow\bbA=\bbA$ follows.
\end{proof}
\isar{Even though $\Downarrow\:\bbA\dist\bbA$ is always a distributor, and so the corresponding functor $F_{\Downarrow}\:\bbA\to\P\bbA$ is well-defined, it need not be the case that $F_{\Downarrow}$ co-restricts to $\Phi_s\bbA$.}
\begin{proposition}\label{prop-cont}\isarnote{Regrouped Propositions, reformulated statements, simplified proofs.}
\isar{For a class of presheaves $\Phi$ and a ${\Q}$-category ${\bbA}$, the following statements are equivalent:
\begin{enumerate}[label=(\arabic*),nosep]
\item $\sup:\Phi_{s}{\bbA}\to {\bbA}$ has a left adjoint,
\item for all $a\in {\bbA}_{0}$, $F_{\Downarrow}(a)\in (\Phi_{s}{\bbA})_{0}$ and $\sup (F_{\Downarrow}(a))\cong a$,
\item $\Downarrow:\bbA\dist\bbA$ is a $\Phi$-distributor that is \emph{approximating}, that is, it satisfies ${\bbA}={\bbA}\swarrow \Downarrow$.
\end{enumerate}
In this case, $F_{\Downarrow}\:\bbA\to\P\bbA$ factors through the full inclusion $\Phi_s\bbA\hookrightarrow\P\bbA$ and the resulting co-restriction $F_{\Downarrow}\:\bbA\to\Phi_s\bbA$ is the left adjoint to $\sup:\Phi_{s}{\bbA}\to {\bbA}$.}
\end{proposition}
\begin{proof}
\isar{($1\Rightarrow 2$) Suppose that $\sup: \Phi_{s}{\bbA}\to {\bbA}$ has a left adjoint $L: {\bbA}\to \Phi_{s}{\bbA}$ in $\Cat(\Q)$, then we can compute that
\begin{align}
\Downarrow\
& = (\sup)^{\natural}\searrow(Y_{\bbA})_{\natural} & \mbox{Definition \ref{way-below}} \\
& = (L)_{\natural}\searrow(Y_{\bbA})_{\natural} & L\dashv\sup\\
& = (L)^{\natural}\circ(Y_{\bbA})_{\natural} & \mbox{Equation \ref{AD}}\\
& = \Phi_s\bbA(Y_{\bbA}-,L-) & \mbox{Equation \ref{AF}} \\
& = \P\bbA(Y_{\bbA}-,L-) & \mbox{full subcategory}
\end{align}
and so, by the Yoneda Lemma, we find for any $a\in\bbA_0$ that
\begin{equation}\label{}
L(a)=\P\bbA(Y_{\bbA}-,L(a))=\ \Downarrow\!(-,a)=F_{\Downarrow}(a).
\end{equation}
This implies that each $F_{\Downarrow}(a)$ is in $\Phi_s\bbA$. Moreover, from $L\dashv\sup\dashv Y_{\bbA}$ and $\sup\circ Y_{\bbA}=1_{\bbA}$ we easily get $\sup\circ L=1_{\bbA}$, which means here that $\sup(F_{\Downarrow}(a))=a$ for all $a\in\bbA_0$. \\
($2\Rightarrow 1$) By hypothesis the functor $F_{\Downarrow}\:\bbA\to\P\bbA$ factors through $\Phi_s\bbA\hookrightarrow\P\bbA$, and the co-restriction $F_{\Downarrow}\:\bbA\to\Phi_s\bbA$ satisfies $\sup\circ F_{\Downarrow}=1_{\bbA}$. A direct computation shows that, for any $\phi\in\Phi_s\bbA$ and $a\in\bbA_0$,
\begin{align}
F_{\Downarrow}(\sup(\phi))(a)
& = \bigwedge_{\psi\in(\Phi_s\bbA)_0}\bbA(\sup(\phi),\sup(\psi))\searrow\psi(a)\\
& \leq \bbA(\sup(\phi),\sup(\phi))\searrow\phi(a) \\
& \leq \phi(a),
\end{align}
that is, $F_{\Downarrow}\circ\sup\leq 1_{\Phi_s\bbA}$. This means that $F_{\Downarrow}\dashv\sup$.\\
($2\Leftrightarrow 3$) By Definition\ \ref{way-below}, $\Downarrow\!(-,a)=F_{\Downarrow}(a)$ for all $a\in\bbA_0$, and so $\Downarrow\:\bbA\dist\bbA$ is a $\Phi$-distributor if and only if $F_{\Downarrow}(a)\in\Phi$ for each $a\in\bbA_0$. Furthermore,}
\begin{equation}\label{}
\sup(F_{\Downarrow}(a))=\colim(F_{\Downarrow}(a),1_{\bbA})=\colim(\Downarrow\!(-,a),1_{\bbA})=\colim(\Downarrow,1_{\bbA})(a)
\end{equation}
and the functor $\colim(\Downarrow,1_{\bbA})\:\bbA\to\bbA$ is characterized by
\begin{equation}\label{}
\colim(\Downarrow,1_{\bbA})_{\natural}=(1_{\bbA})_{\nat}\swarrow \Downarrow=\bbA\swarrow\Downarrow.
\end{equation}
\isar{Therefore we have that $\sup(F_{\Downarrow}(a))$ exists and equals $a$ for all $a\in\bbA_0$ if and only if $\colim(\Downarrow,1_{\bbA})=1_{\bbA}$, if and only if $\bbA=\bbA\swarrow\Downarrow$.}
\end{proof}
\begin{definition}\label{def-cont}
A ${\Q}$-category is said to be \emph{$\Phi$-continuous} if it satisfies the equivalent conditions of Proposition \ref{prop-cont}.
\end{definition}
\begin{proposition}\label{interpolating} If $\Phi$ is a saturated class of presheaves \isar{and $\bbA$ is a} $\Phi$-continuous ${\Q}$-category, \isar{then $\Downarrow\:\bbA\dist\bbA$ is \emph{interpolating}, that is,} $\Downarrow\ =\ \Downarrow\circ \Downarrow$.
\end{proposition}
\begin{proof}\isarnote{Rewritten proof, mentionning explicitly the use of the hypotheses.}\isar{Using (3) of Proposition \ref{prop-cont} we can compute for the $\Phi$-continuous $\bbA$ that
\begin{equation}\label{}
(1_{\bbA})_{\natural}={\bbA}={\bbA}\swarrow \Downarrow=({\bbA}\swarrow \Downarrow)\swarrow \Downarrow={\bbA}\swarrow(\Downarrow\circ\Downarrow)=(1_{\bbA})_{\natural}\swarrow(\Downarrow\circ\Downarrow)
\end{equation}
which shows that $1_{\bbA}=\colim(\Downarrow\circ\Downarrow,1_{\bbA})$. From this we have, for any $a\in\bbA_0$,
\begin{equation}\label{}
a=\colim(\Downarrow\circ\Downarrow,1_{\bbA})(a)=\colim((\Downarrow\circ\Downarrow)(-,a),1_{\bbA})=\sup((\Downarrow\circ\Downarrow)(-,a)).
\end{equation}
Because $\Downarrow:\bbA\to\bbA$ is a $\Phi$-distributor and the class $\Phi$ is saturated, also $\Downarrow\circ\Downarrow$ is a $\Phi$-distributor (cf.\ (3) of Proposition \ref{prop-saturated}). The above thus implies that $(\Downarrow\circ\Downarrow)(-,a)\in\Phi_s\bbA$, and through the adjunction $F_{\Downarrow}\dashv\sup$ we get
\begin{equation}\label{}
\Downarrow\!(-,a)=F_{\Downarrow}(a)\leq(\Downarrow\circ\Downarrow)(-,a).
\end{equation}
This proves that $\Downarrow\ \leq\ \Downarrow\circ\Downarrow$; the reverse inequality was proved in Proposition \ref{prop-1}.}
\end{proof}

\begin{remark}\label{rem-1}
(1) If $\Phi=\mathcal{F}_{in}$, $\Phi=\mathcal{C}^{D}$ or $\Phi=\mathcal{I}_{in}$, then a ($\Phi$-cocomplete) $\Phi$-continuous {\bf 2}-category is just a (directed complete) continuous poset in the sense of \cite{Gie2003}.

(2) If ${\Q}={\bf 2}$, $\Phi$ is a saturated class of presheaves, $Z$ is a union-complete subset system, then a $\Phi$-continuous {\bf 2}-category is just a $Z$-continuous poset in the sense of \cite{Baranga1996}, and a $\Phi$-continuous $\Phi$-compelte {\bf 2}-category is just a $Z$-continuous poset in the sense of \cite{Bandelt1983}.

(3) In the case that ${\Q}$ is a commutative unital quantale, $\Phi$ and $J$ are saturated, the concept of a $\Phi$-continuous ${\Q}$-category agrees with that of a $J$-continuous quantale-enriched category in \cite{HW2011}.

(4) Let ${\Q}$ be a frame and let $\mathcal{F}_{in}$ be the class of presheaves of inhabited flat ideals. Note that the concept of an inhabited flat ideal agree with the concept of an ideal in \cite{LaiZhang2007}. Thus, an $\mathcal{F}_{in}$-continuous $\mathcal{F}_{in}$-cocomplete ${\Q}$-category is just a continuous $\Omega$-partially ordered set or a fuzzy domain in the sense of \cite{LaiZhang2007,Yao2010}.
\end{remark}

\begin{remark}\label{totally-continuous-remark} \isarnote{Simplified and put in a seperate remark.}\isar{The totally continuous $\Q$-categories as defined and studied in \cite{Stubbe2007} correspond (up to the duality referred to in an earlier footnote) with the \emph{$\P$-cocomplete $\P$-continuous $\Q$-categories} in this paper. It is furthermore shown in \cite[Proposition 5.1]{Stubbe2007} that a $\Q$-category $\bbA$ is $\P$-cocomplete and $\P$-continuous if and only if it is the category of ``regular presheaves on a regular $\Q$-semicategory''. We shall not recall this in full generality here, but simply mention a particular case: any presheaf category $\P\bbA$ is $\P$-cocomplete and $\P$-continuous. Indeed, referring to Equations \ref{AI} and \ref{AO}, we know that
\begin{equation}\label{}
F_{\Downarrow}\:\P\bbA\to\P\P\bbA\:\phi\mapsto Y_{\bbA}^{\rightarrow}(\phi)\quad\mbox{ and }\quad\sup\:\P\P\bbA\to\P\bbA\:\Psi\mapsto Y_{\bbA}^{\leftarrow}(\Psi)
\end{equation}
are left/right adjoint to each other.}
\end{remark}
\isar{Given an object $A\in\Q_0$, recall that we write $\{A\}$ for the singleton $\Q$-category whose single hom-arrow is $1_A$, and that $\Q^A$ denotes $\P\{A\}$. It thus follows from Remark \ref{totally-continuous-remark} that ${\Q}^{A}$ is ${\P}$-continuous. To prepare for further examples, we include it here, simplifying somewhat our notation:}
\begin{proposition}\isarnote{More general statement, simpler proof.}
\isar{Let $A$ be any object in any quantaloid ${\Q}$, then ${\Q}^{A}$ is ${\P}$-continuous and the left adjoint to the $\Q$-functor $\sup: {\P}{\Q}^{A}\to {\Q}^{A}$ is given by
\begin{equation}\label{}
d\:\Q^A\to\P\Q^A\:f\mapsto\Big(df\:\Q^A\dist\{\mathrm{dom}(f)\}\Big)\quad\mbox{ where }\quad df(g)=f\circ(1_A\swarrow g).
\end{equation}
If $1_A=\top_{A,A}$ is the top element of $\Q(A,A)$, then this simplifies to $df(g)=f\circ\top_{\mathrm{cod} (g),A}$.}
\end{proposition}
\begin{proof}
\isar{It is a matter of making explicit the formula in Remark \ref{totally-continuous-remark} for $\bbA=\{A\}$. An object of $\Q^A$ is some morphism $f\:A\to X$ in $\Q$. The Yoneda embedding $Y_{\{A\}}\:\{A\}\mapsto\Q^A$ sends the single object of $\{A\}$ to $1_A$. Therefore
\begin{equation}\label{}
df:=F_{\Downarrow}(f)=Y_{\{A\}}^{\rightarrow}(f)=f\circ\Q^A(-,1_A)=f\circ(1_A\swarrow -)
\end{equation}
as claimed. For any $g\:A\to Y$ we have
\begin{equation}\label{}
\top_{Y,A}\leq 1_A\swarrow g\quad\Longleftrightarrow\quad\top_{Y,A}\circ g\leq 1_A,
\end{equation}
the latter of which holds whenever $1_A=\top_{A,A}$; thus we get $\top_{Y,A}=1_A\swarrow g$ in this case.}
\end{proof}
\isar{In a similar way we find:}
\begin{proposition}
\isar{Let $A$ be any object in a quantaloid ${\Q}$ for which $1_{A}=\top_{A, A}$.} If, for every morphism $f\:A\to X$ and object $B$ in ${\Q}$, the map $f\circ-:{\Q}(B, A)\to {\Q}(B,X)$ preserves finite meets, then ${\Q}^{A}$ is $\mathcal{F}$-continuous.
\end{proposition}
\begin{proof}
Let $\lambda:\{q\}\dist{\Q}^{A}$ be \isar{any} copresheaf. Then for $f\in {\Q}(A, X)$ and $g\in {\Q}(A, Y)$ we have ${\Q}^{A}(f, g)\circ \lambda(f)\leq \lambda (g)$, i.e., $g\swarrow f\leq \lambda(g)\swarrow \lambda(f)$. Thus, $\top_{Y, A}\leq \top_{A, A}\swarrow g\leq \lambda(\top_{A, A})\swarrow \lambda(g)$, whence $\top_{Y, A}\circ \lambda(g)\leq \lambda(\top_{A, A})$. Furthermore, if $g\in{\Q}(A, A)$, then $\lambda(g)\leq \lambda(1_{A})$.

\isar{Now, f}or $f\in ({\Q}^{A})_{0}$ and $\lambda_{1}, \lambda_{2}\in({\P}^{\dag}{\Q}^{A})_{0}$ with the same type $q$, we have \begin{align} (df\circ\lambda_{1})\wedge (df\circ\lambda_{2})&=\left(\bigvee_{g\in({\Q}^{A})_{0}}df(g)\circ \lambda_{1}(g)\right)\wedge\left(\bigvee_{g\in({\Q}^{A})_{0}}df(g)\circ \lambda_{2}(g)\right)\\ &=\left(\bigvee_{g\in({\Q}^{A})_{0}}f\circ\top_{\mathrm{cod}(g), A}\circ \lambda_{1}(g)\right)\wedge\left(\bigvee_{g\in({\Q}^{A})_{0}}f\circ\top_{\mathrm{cod}(g), A}\circ \lambda_{2}(g)\right)\\ &=f\circ\left(\bigvee_{g\in({\Q}^{A})_{0}}\top_{\mathrm{cod}(g), A}\circ \lambda_{1}(g)\right)\wedge f\circ\left(\bigvee_{g\in({\Q}^{A})_{0}}\top_{\mathrm{cod}(g), A}\circ \lambda_{2}(g)\right)\\ &= (f\circ\lambda_{1}(1_{A}))\wedge (f\circ\lambda_{2}(1_{A}))
\end{align}
and
\begin{align}
df\circ (\lambda_{1}\wedge\lambda_{2})
&=\bigvee_{g\in({\Q}^{A})_{0}}df(g)\circ(\lambda_{1}\wedge\lambda_{2})(g)\\
&\geq df(1_{A})\circ (\lambda_{1}\wedge\lambda_{2})(1_{A})\\
&=f\circ (\lambda_{1}(1_{A})\wedge\lambda_{2}(1_{A})).
\end{align}
By the above arguments, it is easy to see that $(df\circ\lambda_{1})\wedge (df\circ\lambda_{2})=df\circ (\lambda_{1}\wedge\lambda_{2})$. Thus $df$ is a flat ideal for every $f\in ({\Q}^{A})_{0}$. Therefore ${\Q}^{A}$ is $\mathcal{F}$-continuous.
\end{proof}

\isarnote{Reformulated, put the historical references right.}\isar{For the next Proposition, we recall from \cite{Hohle2011,Tao2014}\isarnote{} that a (unital) quantale $(Q,\bigvee,\&,1)$, in which we shall write the residuations as $(a\&-)\dashv(a\backslash-)$ and $(-\&a)\dashv(-/a)$, is said to be \emph{divisible} when
\begin{equation}\label{}
a\&(a\backslash b)=a\wedge b=(b/a)\&a
\end{equation}
for all $a,b\in Q$. As pointed out in the cited references, any such divisible quantale enjoys several properties, notably:
\begin{enumerate}[label=(\roman*),nosep]
\item $Q$ is \emph{integral}, that is, $1=\top$,
\item $Q$ is \emph{localic}, that is, $a\wedge\bigvee_jb_j=\bigvee_j(a\wedge b_j)$ for all $a,(b_j)_j\in Q$,
\item $(b\wedge c)\&a=(b\&a)\wedge(c\&a)$ and $a\&(b\wedge c)=(a\&b)\wedge (a\&c)$ for all $a,b,c\in Q$.
\end{enumerate}
As Hohle and Kubiak \cite{Hohle2011} first pointed out, there is an important \emph{quantaloid} $B_Q$ associated with a divisible quantale, as follows:}
\begin{enumerate}[label=-,nosep]
\item objects: are the elements $x,y,z,...$ of $Q$,
\item hom-lattices: $B_{Q}(x,y):=\{a\in Q \mid a\leq x\wedge y\}$ (with order inherited from $Q$),
\item composition: for $a\in B_{Q}(x,y)$ and $b\in B(y,z)$, $b\circ a:=b\&(y\backslash a) = (b/y )\& a$,
\item identities: $1_{x}:=x$.
\end{enumerate}
\isar{Categories enriched over this quantaloid can be regarded as fuzzy sets endowed with fuzzy preorders \cite{Hohle2011,Tao2014}, and when applied to the divisible quantale $([0,\infty],\bigwedge,+,0)$ of Example 2.7(3), those enriched categories are (generalized) partial metric spaces \cite{Hohle2011,Stubbe2014,HS2018}. (In a suitable context, the construction of $B_Q$ from $Q$ has a universal property, see \cite[Example 2.14]{Stubbe2014}.)}

\begin{proposition}\isarnote{Slightly changed notation from greek to roman letters, to better suit the notation above.}
Let $(Q,\isar{\bigvee,\&,1})$ be a \isar{unital divisible quantale and $B_{Q}$ be the quantaloid described above}. Then for every object $A$ in $B_{Q}$, $(B_{Q})^{A}$ is $\mathcal{F}$-continuous.
\end{proposition}
\begin{proof}
\isar{The quantaloid $B_Q$ being integral (i.e.\ $1_x=x$ is the top element of $B_Q(x,x)$), we only need to show that, for any $b\in B_Q(y,z)$ and $x\in(B_Q)_0$, the sup-morphism $b\circ-\:B_Q(x,y)\to B_Q(x,z)$ preserves finite meets. For $a_1,a_2\in B_Q(x,y)$ we compute:
$b\circ(a_{1}\wedge a_{2})
=b\&(y\backslash(a_{1}\wedge a_{2}))
=b\&(y\backslash a_{1}\wedge y\backslash a_{2})
=b\&(y\backslash a_{1})\wedge b\&(y\backslash a_{2})
=b\circ a_{1}\wedge b\circ a_{2}$.}
\end{proof}
\begin{example} \isar{Following \cite[Examples 2.7 and 2.11]{Stubbe2014}, for any set $X$ we denote by} ${\Q}(X)$ the quantaloid whose objects are the subsets of $X$, and in which the arrows from $S\subseteq X$ to $T\subseteq X$ are precisely all $U\subseteq S\cap T$.\isarnote{Made a single reference to \cite{Stubbe2014}.} The composition law in ${\Q}(X)$ is given by intersection, and the identity on the object $S\subseteq X$ is $S:S\to S$ itself. Using ${\Q}(X)$-enriched categories, we can overcome the issue of partial elements in a quantale-enriched category. \isar{In fact, it is easy to see that the quantale $({\P}(X),\bigcup,\cap,X)$ of subsets of $X$ is divisible, and that the quantaloid ${\Q}(X)$ is precisely $B_{{\P}(X)}$.} Therefore, for every object $S$ in ${\Q}(X)$, the ${\Q}(X)$-enriched category ${\Q}(X)^{S}$ is $\mathcal{F}$-continuous.
\end{example}
\isarnote{Moved up the end of this section to earlier position.}

\section{Algebraicity in ${\Q}$-categories}\label{S6}

In this section \isar{we sketch the appropriate generalization of \cite[Section 6]{Stubbe2007}: we define and study (approximation by) $\Phi$-compact elements in ${\Q}$-categories.}

\isar{Recall from Proposition \ref{prop-1} that, for any class of presheaves $\Phi$ and any $\Q$-category $\bbA$, we always have a $\Q$-functor $F_{\Downarrow}\:\bbA\to\P\bbA$ which satisfies $F_{\Downarrow}\leq Y_{\bbA}$. We now study when this inequality is actually an equality.}
\begin{proposition}\label{prop-2}\isarnote{Turned Remark into Proposition and moved it up. (There was a mistake in the Remark: for general $\phi\in\Phi\bbA$ there is no $\sup$. Replaced it with correct condition.)}
\isar{Let ${\bbA}$ be a ${\Q}$-category and $\Phi$ a class of presheaves on $\Cat(\Q)$. For $a\in {\bbA}_{0}$, the following condititions are equivalent:
\begin{enumerate}[label={(\arabic*)},nosep]
\item $Y_{\bbA}(a)\leq F_{\Downarrow}(a)$,
\item ${\bbA}(x,a)\leq\ \Downarrow\!(x, a)$ for every $x\in {\bbA}_{0}$,
\item ${\bbA}(a,x)\leq\ \Downarrow\!(a, x)$ for every $x\in {\bbA}_{0}$,
\item $1_{ta}\leq\ \Downarrow\!(a, a)$.
\item ${\bbA}(a,\sup(\phi))\leq \phi(a)$ for all $\phi\in(\Phi_{\isar{s}}{\bbA})_{0}$.
\end{enumerate}
In all but the fourth condition, the inequality can be replaced by an equality.}
\end{proposition}
\begin{proof}\isarnote{Added a short proof.}\isar{
($1\Leftrightarrow 2$) Is a tautology.\\
($2\Rightarrow 4$ and $3\Rightarrow 4$) Put $x=a$ and use that $1_{ta}\leq\bbA(a,a)$.\\
($4\Rightarrow 2$ and $4\Rightarrow 3$) Follows directly from the distributor inequalities
\begin{equation}\label{}
\Downarrow\!(a,a)\circ\bbA(x,a)\leq\ \Downarrow\!(x,a)\mbox{ and }\ \bbA(a,x)\circ\Downarrow\!(a,a)\leq\ \Downarrow\!(a,x).
\end{equation}
($4\Leftrightarrow 5$) Follows directly from an explicit computation of $\Downarrow$ with Definition 5.1.}
\end{proof}
\begin{definition}\label{def-compact}
Let ${\bbA}$ be a ${\Q}$-category and $\Phi$ a class of presheaves on $\Cat(\Q)$. We say that $a\in {\bbA}_{0}$ is $\Phi$-compact if the equivalent condititions in Proposition \ref{prop-2} hold.
\end{definition}
\isar{There are non-trivial examples of such compact elements:}
\begin{proposition}\label{prop-repr}
Let $\Phi$ be a saturated class of presheaves on $\Cat(\Q)$. For any $\Q$-category ${\bbA}$ and any $x\in {\bbA}_{0}$, the \isar{representable presheaf} $Y_{\bbA}(x)$ is $\Phi$-compact in $\Phi{\bbA}$.
\end{proposition}
\begin{proof}\isarnote{Put slightly more detail in the argument, and put it in "align" for clarity.}
By Proposition \ref{prop-saturated} we know that $\Phi\bbA$ is $\Phi$-cocomplete, and that the full embedding $i_{\bbA}\:\Phi\bbA\to\P\bbA$ is $\Phi$-cocontinuous. For clarity, we shall write $Y'_{\bbA}\:\bbA\to\Phi\bbA$ for the factorisation of the Yoneda embedding $Y_{\bbA}\:\bbA\to\P\bbA$ through $i_{\bbA}$. Then we can compute for any $\Psi\in(\Phi(\Phi\bbA))_0$ and $x\in\bbA_0$ that
\begin{align}
\Phi{\bbA}(Y'_{\bbA}(x),\sup(\Psi))
&=\P\bbA(i(Y'_{\bbA}(x)),i(\sup(\Psi))) & \mbox{full embedding} \\
&=\P\bbA(Y_{\bbA}(x),\sup(i_{\bbA}^{\rightarrow}(\Psi))) & \mbox{$\Phi$-cocontinuity}\\
&=\sup(i_{\bbA}^{\rightarrow}(\Psi))(x) & \mbox{Equation \ref{YonedaLemma}}\\
&=(i_{\bbA}^{\rightarrow}(\Psi))(Y_{\bbA}(x)) & \mbox{Equation \ref{AO}}\\
&=\Psi\circ {\P}{\bbA}(Y_{\bbA}(x), i-) & \mbox{definition of $i_{\bbA}^{\rightarrow}$} \\
&=\Psi\circ\Phi\bbA(Y'_{\bbA}(x),-) & \mbox{full embedding} \\
&\leq\Psi(Y'_{\bbA}(x)) & \mbox{distributor inequality}
\end{align}
The result follows from (5) in Proposition \ref{prop-2}.
\end{proof}
\isarnote{Reformulated slightly. There is already a functor $i$ in Prop. 3.8, so call it $j$ here.}
\isar{Writing $\bbA_c$ for the full subcategory determined by the $\Phi$-compact elements of $\bbA$, it is clear from the definition that the full inclusion  $j_{\bbA}\:\bbA_c\hookrightarrow\bbA$ is the \emph{inverter} \cite{bkps89} of the 2-cell
\begin{equation}\label{}
\bbA\xymatrix@C=8ex{
\ar@/^1ex/@<0.5ex>[r]^{Y_{\bbA}}\ar@{}[r]|{\rotatebox{90}{$\leq$}}\ar@/_1ex/@<-0.5ex>[r]_{F_{\Downarrow}} &}\P\bbA
\end{equation}
in $\Cat(\Q)$. We furthermore define the ${\Q}$-distributor
\begin{equation}\label{}
\Sigma_{\bbA}:=(j_{\bbA})_{\natural}\circ (j_{\bbA})^{\natural}:{\bbA}\dist{\bbA},
\end{equation}
and use the classifying property of $\P\bbA$ to define the functor
\begin{equation}\label{}
S_{\bbA}:{\bbA}\to {\P}{\bbA}\:a\mapsto\Sigma_{\bbA}(-,a)
\end{equation}
too. By Propositions \ref{prop-2} and \ref{prop-1} it easily follows that
\begin{equation}\label{}
\Sigma_{\bbA}(x,y)=\bigvee_{a\in\bbA_c}\bbA(a,y)\circ\bbA(x,a)\leq \bigvee_{a\in\bbA_c}\Downarrow\!(a,y)\ \circ\Downarrow\!(x,a)\leq\ \Downarrow\!(x,y),
\end{equation}
that is $\Sigma_{\bbA}\leq\ \Downarrow$ in $\Dist(\Q)$, and thus $S_{\bbA}\leq F_{\Downarrow}$ in $\Cat(\Q)$ too. }
\begin{definition}
Let $\Phi$ be a class of presheaves on $\Cat(\Q)$. We say that a $\Q$-category $\bbA$ is \emph{$\Phi$-algebraic} if, for every $x\in {\bbA}_{0}$, $S_{\bbA}(x)\in (\Phi_{s}{\bbA})_{0}$ and $x\cong\sup (S_{\bbA}(x))$.
\end{definition}
\begin{remark}
By \isar{a} similar analysis \isar{as} in Remarks \ref{rem-1} and \ref{totally-continuous-remark}, we can see that algebraic posets \cite{Gie2003}, $Z$-algebraic posets \cite{Baranga1996}, algebraic $\Omega$-categories \cite{Yao2014} and totally algebraic ${\Q}$-categories \cite{Stubbe2007} are \isar{particular cases of (skeletal)} $\Phi$-algebraic ${\Q}$-categories.
\end{remark}
\begin{proposition}\label{prop-alg}\isarnote{Removed hypothesis that $\bbA$ is $\Phi$-cocomplete: one only has to take the $\bigwedge$ over $(\Phi_s\bbA)_0$.}
\isar{Let $\Phi$ be a class of presheaves on $\Cat(\Q)$.} A $\Q$-category ${\bbA}$ is $\Phi$-algebraic \isar{if and only if} ${\bbA}$ is $\Phi$-continuous and $S_{\bbA}=F_{\Downarrow}$.
\end{proposition}
\begin{proof}
Sufficiency is clear. \isar{Conversely, supposing that ${\bbA}$ is $\Phi$-algebraic, we can compute, for any $x,y\in\bbA_0$,
\begin{align}
\Downarrow\!(y,x)
&=\bigwedge_{\phi\in(\Phi_\isar{s}{\bbA})_{0}}{\bbA}(x, \sup(\phi))\searrow \phi(y)\\
&\leq {\bbA}(x,\sup(S_{\bbA}(x)))\searrow S_{\bbA}(x)(y) & \mbox{$S_{\bbA}(x)\in(\Phi_s\bbA)_0$}\\
&\leq {\bbA}(x,x)\searrow S_{\bbA}(x)(y) & \mbox{$\sup(S_{\bbA}(x))=x$}\\
&\leq 1_{tx}\searrow \Sigma_{\bbA}(y,x)&\mbox{$1_{tx}\leq\bbA(x,x)$}\\
&=\Sigma_{\bbA}(y,x)
\end{align}
Thus $F_{\Downarrow}\leq S_{\bbA}$, and since the converse inequality always holds, the result follows.}
\end{proof}
\begin{proposition}\label{prop-alg-2}
Let $\Phi$ be a saturated class of presheaves on $\Cat(\Q)$. Then for any $\Q$-category $\bbA$, $\Phi{\bbA}$ is \isar{($\Phi$-cocomplete and)} $\Phi$-algebraic.
\end{proposition}
\begin{proof}
\isar{Making use of Proposition \ref{prop-saturated}, we denote by $Y'_{\bbA}:{\bbA}\to\Phi{\bbA}$ the factorisation of the Yoneda embedding $Y_{\bbA}\:\bbA\to\P\bbA$ through the full embedding $i_{\bbA}:\Phi{\bbA}\to{\P}{\bbA}$. For the functor $(Y'_{\bbA})^{\rightarrow}\:\Phi\bbA\to\Phi(\Phi\bbA)$ and any $\phi\in (\Phi{\bbA})_{0}$ we may then compute that}\isarnote{Slightly rewritten computation.}
\begin{align}
\sup\Big((Y'_{\bbA}){}^{\to}(\phi)\Big)
&=\sup\Big(((i_{\bbA})^{\to}\circ Y'_{\bbA}{}^{\to})(\phi)\Big)&\mbox{Proposition \ref{prop-saturated}}\\
&=\sup\Big((Y_{\bbA})^{\to}(\phi)\Big) & \mbox{Equation \ref{AJ}}\\
&=\sup\Big(\phi\circ(Y_{\bbA})^\nat\Big) & \mbox{definition of $(Y_{\bbA})^{\to}$}\\
&=\phi\circ(Y_{\bbA})^{\natural}\circ(Y_{\bbA})_{\natural}&\mbox{Equation \ref{AO}}\\
&= \phi\circ\bbA&\mbox{Equation \ref{YonedaFF}}\\
&=\phi.
\end{align}
\isar{Similarly, for any $\Psi\in(\Phi(\Phi\bbA))_0$, we can spell out that\isarnote{A slightly different argument.}
\begin{align}
(Y'_{\bbA}){}^{\to}\Big(\sup(\Psi)\Big)
& = \Psi\circ(i_{\bbA})^\nat\circ(Y_{\bbA})_\nat\circ(Y'_{\bbA})^\nat & \mbox{as above}\\
& = \Psi\circ(i_{\bbA})^\nat\circ(i_{\bbA}\circ Y'_{\bbA})_\nat\circ(Y'_{\bbA})^\nat & \mbox{$Y_{\bbA}=i_{\bbA}\circ Y'_{\bbA}$}\\
& = \Psi\circ(i_{\bbA})^\nat\circ (i_{\bbA})_\nat\circ(Y'_{\bbA})_\nat\circ(Y'_{\bbA})^\nat & \mbox{Equation \ref{AE}}\\
& = \Psi\circ(Y'_{\bbA})_\nat\circ(Y'_{\bbA})^\nat & \mbox{$i_{\bbA}$ is fully faithful}\\
& \leq \Psi. & \mbox{$(Y'_{\bbA})_\nat\dashv (Y'_{\bbA})^\nat$}
\end{align}
That is, $(Y'_{\bbA})^\to\:\Phi\bbA\to\Phi(\Phi\bbA)$ is left adjoint to $\sup\:\Phi(\Phi\bbA)\to\Phi\bbA$. By Proposition \ref{prop-cont}, $\Phi\bbA$ is $\Phi$-continuous and $F_{\Downarrow}=(Y'_{\bbA})^\to$ on $\Phi\bbA$.} \isarnote{Slightly different notations.} Finally, for any $\phi_{1},\phi_{2}\in (\Phi{\bbA})_{0}$ we compute that
\begin{align}
F_{\Downarrow}(\phi_1)(\phi_2)
& = \Big((Y'_{\bbA})^\to(\phi_1)\Big)(\phi_2) & \mbox{shown above}\\
& = \Big(\phi_1\circ(Y'_{\bbA})^\nat\Big)(\phi_2) & \mbox{definition of $(Y'_{\bbA})^\to$} \\
& = \Big(\bigvee_{x\in {\bbA}_{0}} \phi_{1}(x)\circ\Phi{\bbA}(-,Y'_{\bbA}(x))\Big)(\phi_2) & \mbox{definition of $(Y'_{\bbA})^\nat$} \\
&=\bigvee_{x\in {\bbA}_{0}} \phi_{1}(x)\circ\Phi{\bbA}(\phi_{2}, Y'_{\bbA}(x)) \\
&=\bigvee_{x\in {\bbA}_{0}} \Phi{\bbA}(Y'_{\bbA}(x), \phi_{1})\circ\Phi{\bbA}(\phi_{2}, Y'_{\bbA}(x)) & \mbox{Equation \ref{YonedaLemma}}\\
&\leq\bigvee_{\kappa\in((\Phi{\bbA})_{c})_{0}}\Phi\bbA(\kappa,\phi_{2})\circ\Phi{\bbA}(\phi_{1}, \kappa) & \mbox{Proposition \ref{prop-repr}}\\
&=\Sigma_{\Phi{\bbA}}(\phi_{2}, \phi_{1})\\
&=S_{\bbA}(\phi_1)(\phi_2)
\end{align}
\isar{So we find that $F_{\Downarrow}\leq S_{\bbA}$ on $\Phi\bbA$, and since the inverse inequality always holds, we find $F_{\Downarrow}= S_{\bbA}$. By Proposition \ref{prop-alg}, $\Phi\bbA$ is $\Phi$-algebraic.}
\end{proof}
\begin{proposition}\label{prop-alg-3}\isarnote{Equivalence instead of isomorphism.}
\isar{Let $\Phi$ be a class of presheaves on $\Cat(\Q)$.} If ${\bbA}$ is a $\Phi$-cocomplete and $\Phi$-algebraic ${\Q}$-category, then ${\bbA}\simeq\Phi({\bbA}_{c})$.
\end{proposition}
\begin{proof}\isarnote{Slightly changed notations and presentation.}
For the full embedding $j_{\bbA}:{\bbA}_{c}\to {\bbA}$, the co-restriction $Y'_{\bbA}\:\bbA\to\Phi\bbA$ of the Yoneda embedding, and its left adjoint $\sup\:\Phi\bbA\to\bbA$, consider the composite functors
\begin{equation}\label{}
\xymatrix{
& \Phi\bbA\ar[dr]^{(j_{\bbA})^{\leftarrow}} \\
\bbA\ar[ur]^{Y'_{\bbA}}\ar@{-->}[rr]_F & & \Phi(\bbA_c)}
\quad
\xymatrix{
& \Phi\bbA\ar[dr]^{\sup} \\
\Phi(\bbA_c)\ar[ur]^{(j_{\bbA})^\to}\ar@{-->}[rr]_G & & \bbA}
\end{equation}
Then we have, for any $x\in\bbA_0$,
\begin{align}
(G\circ F)(x)
&=(\sup\circ(j_{\bbA})^\to\circ(j_{\bbA})^{\leftarrow}\circ Y'_{\bbA})(x)\\
&=\sup\Big(\bbA(-,x)\circ (j_{\bbA})_\nat\circ(j_{\bbA})^\nat\Big)\\
&=\sup(\Sigma_\bbA(-,x))\\
&=\sup(S_{\bbA}(x))\\
&\cong x,
\end{align}
so $G\circ F\cong 1_{\bbA}$. Similarly, for any $\phi\in\Phi(\bbA_c)$,
\begin{align}
(F\circ G)(\phi)
&=((j_{\bbA})^{\leftarrow}\circ Y'_{\bbA}\circ\sup\circ(j_{\bbA})^\to)(\phi)\\
&= Y'_{\bbA}\Big(\sup((j_{\bbA})^\to (\phi))\Big)\circ (j_{\bbA})_\nat \\
&=\bbA\Big(-,\sup((j_{\bbA})^\to (\phi))\Big)\circ\bbA(j_{\bbA}-,-)\\
&=\bbA\Big(j_{\bbA}-,\sup((j_{\bbA})^\to (\phi))\Big)\\
&=\phi,
\end{align}
(the last equality follows from Proposition \ref{prop-2} (5)) and so $G\circ F=1_{\Phi(\bbA_c)}$ too. This gives the required equivalence.
\end{proof}
By Propositions 6.7 and 6.8 we can conclude that:
\begin{theorem}\label{thm-alg}
Let $\Phi$ be a saturated class of presheaves on $\Cat(\Q)$. Then a $\Q$-category ${\bbA}$ is $\Phi$-cocomplete and $\Phi$-algebraic if and only if ${\bbA}\simeq\Phi{\bbB}$ for some ${\Q}$-category ${\bbB}$.
\end{theorem}
\isar{If $\bbA$ is a skeletal $\Q$-category, then the equivalence in the above Theorem is in fact an isomorphism.}

\section{Conclusion and some further work}\label{S7}

In this paper, we develop some basic concepts of continuous ${\Q}$-categories. We mainly proposed three kinds of saturated classes of presheaves. Not only are they proper generalizations of the concept of ideals in domain theory but also they have different properties. This may provide us an useful tool to develop a theory of dynamic domains. Up to now, the concept of continuity in poset\isar{s} has undergone a series of generalizations\isar{, as indicated in the} following diagram. (Here the concepts of continuous posets, $Z$-continuous posets and $J$-continuous quantale-enriched categories are in the sense of \cite{Gie2003}, \cite{Baranga1996} and \cite{HW2011}, respectively.)

\begin{picture}(100,200)(0,0) \put(140,160){\framebox(100,20){Continuous posets}} \put(190,155){\vector(0,-1){20.0}}

\put(130,110){\framebox(110,20){$Z$-continuous posets}} \put(190,105){\vector(0,-1){20.0}}

\put(250,110){\framebox(60,20){$Z$-level}}

\put(20,60){\framebox(220,20){$J$-continuous quantale-enriched categories}} \put(190,55){\vector(0,-1){20.0}}

\put(250,60){\framebox(110,20){Many-valued level}}

\put(10,10){\framebox(230,20){$\Phi$-continuous quantaloid-enriched categories}}

\put(250,10){\framebox(180,20){Many-typed and many-valued level}} \end{picture}

With the rapid development of related disciplines, there are many aspects of continuous quantaloid-enriched categories \isar{that} deserve further study:

(1) Completely distributive quantaloid-enriched categories are developed in \cite{Stubbe2007}. In classic domain theory, we know that there exist close relations between the category of continuous dcpos, the category of posets and the category of completely distributive lattices. Thus, we can investigate categorical relations between continuous quantaloid-enriched categories and some other structures such as completely distributive quantaloid-enriched categories.

(2) Exponentiable functors between quantaloid-enriched categories are studied in \cite{CleHof2009}. Turning to the continuous setting, this problem seems to be more complicated. Thus, how to characterize exponentiable functors between continuous quantaloid-enriched categories is an interesting problem.

(3) A theory of quantaloid-enriched topological spaces was developed by H\"{o}hle in \cite{Hohle2014}, which forms a common framework for many-valued topology as well as for non-commutative topology. This may give us an inspiration to study the topological aspects of continuous quantaloid-enriched categories.

Domain theory \isar{is} also studied in the framework of order-enriched categories, based on the concept of Kock-Z\"{o}berlein monads, by Hofmann and Sousa \cite{Hofmann2017}. These works can be viewed as generalizations of Domain Theory to categorical levels in different directions and they may provide inspirations to each other.

\section*{Acknowledgement}

This work is supported by the National Natural Science Foundation of China (Grant Nos.11871320, 11501048), Natural Science Basic Research Program of Shaanxi (Program No. 2022JM-032) and the Fundamental Research Funds for the Central Universities (Grant No. 300102122109). \isar{The authors thank the Referee for valuable suggestions.}

\end{document}